\documentclass[]{article}

\usepackage{tabularx} %
\usepackage{amsmath}
\usepackage{mleftright}
\usepackage{amsthm}
\usepackage{xypic}
\usepackage{latexsym}
\usepackage[plain,vlined]{algorithm2e}
\usepackage{changepage}
\usepackage{amssymb}
\usepackage{graphicx}
\usepackage[margin=1.2in,letterpaper]{geometry} %
\usepackage{cite}
\usepackage[final]{hyperref}
\usepackage{comment}
\usepackage{enumitem}
\usepackage{tikz}
\usetikzlibrary{arrows,automata,positioning}
\usetikzlibrary{arrows.meta}

\usepackage{subcaption}

\newtheorem{theorem}{Theorem}[section]
\newtheorem{corollary}[theorem]{Corollary}
\newtheorem{defi}[theorem]{Definition}
\newtheorem{prop}[theorem]{Proposition}
\newtheorem{lemma}[theorem]{Lemma}
\newtheorem{remark}[theorem]{Remark}

\usepackage{xcolor}

\newcommand{\bx}{\boldsymbol{x}}
\newcommand{\bw}{\boldsymbol{w}}
\newcommand{\bv}{\boldsymbol{v}}
\newcommand{\bH}{\boldsymbol{\mathcal{H}}}
\newcommand{\bD}{\boldsymbol{\mathcal{D}}}
\newcommand{\dom}{\text{dom}}
\newcommand{\MM}{\mathfrak{M}}

\newcommand{\cA}{\mathcal{A}}
\newcommand{\cM}{\mathcal{M}}
\newcommand{\cC}{\mathcal{C}}
\newcommand{\cZ}{\mathcal{Z}}
\newcommand{\cE}{\mathcal{E}}
\newcommand{\cN}{\mathcal{N}}
\newcommand{\vertiii}[1]{{\left\vert\kern-0.25ex\left\vert\kern-0.25ex\left\vert #1
			\right\vert\kern-0.25ex\right\vert\kern-0.25ex\right\vert}}

\DeclareMathOperator{\zer}{zer}
\DeclareMathOperator{\Fix}{Fix}
\DeclareMathOperator{\Img}{Im}

\DeclareMathOperator{\Span}{span}
\DeclareMathOperator{\Var}{Var}
\DeclareMathOperator{\Div}{div}

\title{\textbf{Graph and distributed extensions of the Douglas--Rachford method}}
\author{Kristian Bredies\thanks{Institute of Mathematics and Scientific Computing, University of Graz, Graz, Austria. email: \texttt{kristian.bredies@uni-graz.at}, \texttt{enis.chenchene@uni-graz.at}.} \and Enis Chenchene\footnotemark[1] \and Emanuele Naldi\thanks{Institute of Analysis and Algebra, TU Braunschweig, email:  \texttt{e.naldi@tu-braunschweig.de}}}
\date{\today}

\begin{document}

\maketitle
\begin{abstract}
	In this paper, we propose several graph-based extensions of the Douglas--Rachford splitting (DRS) method to solve monotone inclusion problems involving the sum of $N$ maximal monotone operators. Our construction is based on a two-layer architecture that we refer to as bilevel graphs, to which we associate a generalization of the DRS algorithm that presents the prescribed structure. The resulting schemes can be understood as unconditionally stable frugal resolvent splitting methods with a minimal lifting in the sense of Ryu [Math Program 182(1):233--273, 2020], as well as instances of the (degenerate) Preconditioned Proximal Point method, which provides robust convergence guarantees. We further describe how the graph-based extensions of the DRS method can be leveraged to design new fully distributed protocols. Applications to a congested optimal transport problem and to distributed Support Vector Machines show interesting connections with the underlying graph topology and highly competitive performances with state-of-the-art distributed optimization approaches.
\end{abstract}

\section{Introduction}
The proximal point algorithm is a widely used tool for solving a variety of problems such as finding zeros of maximal monotone operators, fixed-points of nonexpansive mappings, as well as minimizing convex functions. Given a Hilbert space $\bH$, the \emph{Preconditioned Proximal Point} (PPP) method can be understood as proximal point method with respect to a new metric induced by a self-adjoint (uniformly) positive definite linear map $\mathcal{M}:\bH\to \bH$. For a maximal monotone operator $\mathcal{A}:\bH\to 2^{\bH}$, the general iteration of a PPP method reads
\begin{equation}\label{eq:PPP}
	u^0 \in \bH, \quad u^{k+1}=u^k+\theta_k\left(\mathcal{T}u^k -u^k\right) \quad \text{for all} \ k\in \mathbb{N},
\end{equation}
where $\mathcal{T}:=\left( \mathcal{M}+\mathcal{A}\right)^{-1}\mathcal{M}$, and $\theta_k\in (0,2]$ are relaxation parameters that satisfy $\sum_k\theta_k(2-\theta_k) = +\infty$. In our recent work \cite{bredies2021degenerate}, we focused on the degenerate case, i.e., assuming that $\mathcal{M}$ is only positive semidefinite, allowing in this way $\mathcal{M}$ to have a possibly large kernel. In that case, for the iterations in \eqref{eq:PPP} to make sense, we restricted the analysis to the class of \emph{admissible preconditioners}, i.e., such that $\mathcal{T}$ is everywhere defined and single-valued. The degenerate PPP framework \cite{bredies2021degenerate, BrediesDRS, Bredies2017APP} allows us to study in a unifying theory a large class of known (and new) splitting methods such as Chambolle--Pock \cite{Chambolle2011} (also in the degenerate case, i.e., in the notation of \cite{Chambolle2011}, where $\tau\sigma L^2=1)$, Peaceman--Rachford \cite{peaceman_rachford}, Davis--Yin \cite{Davis2016} and Douglas--Rachford \cite{drs_mercier_lions}, and to easily derive new extensions to the $N$-operator problem:
	\begin{equation}\label{eq:Nop}
		\text{find} \ x \in H \ \text{such that:} \ 0 \in (A_1+\cdots+A_N)x,
	\end{equation}
	where $A_i$ are maximal monotone operators on the Hilbert space $H$. To solve problem \eqref{eq:Nop}, which we will always assume possible, we consider the class of so-called \emph{frugal resolvent splitting} (FRS) methods introduced by Ryu in \cite{Ryu}. These are iterative methods, which at every iteration only require a single evaluation of the resolvents, i.e., $J_{\sigma_i A_i}:=(I+\sigma_i A_i)^{-1}$ for some $\sigma_i>0$, and simple algebraic operations, such as vector additions and scalar multiplications. It has also been proven in \cite{Ryu} and later extended in \cite{malitsky2021resolvent} that if $N>2$, \emph{unconditionally stable} FRS methods, i.e., which produce (weakly) convergent sequences to a solution to \eqref{eq:Nop} for every tuple $(A_1, \dots, A_N)$ of maximal monotone operators, can only be designed on a $d$-fold product space with $d\geq N-1$, thus requiring several additional variables. FRS methods with $d=N-1$ are said to have a \emph{minimal lifting} or \emph{minimal variables}. In this paper, we focus on FRS methods with minimal variables.

	While for the two-operator case, i.e., with $N=2$, the class of unconditionally stable FRS methods with minimal variables reduces to the celebrated DRS method \cite[Corollary 1]{Ryu}, for larger problems, the resulting schemes present many different structures. Unconditionally stable FRS schemes with minimal variables and parallel structures can be derived with the so-called \textit{product-space trick}, see, e.g., \cite[Section 9.1]{condat}. Schemes with different structures have been discovered more recently. The Sequential DRS, introduced in \cite{bredies2021degenerate}, presents a purely sequential nature, which is very close to the method introduced by Malitsky and Tam in \cite{malitsky2021resolvent}, where the pure sequentiality is in some sense broken with an additional communication between the first and the last operator. The Malitsky--Tam splitting, when $N=3$, is in turn highly related, yet not equivalent, to the method introduced by Ryu in \cite{Ryu}. This systematic unfolding of structurally different FRS methods with minimal variables leads us to the natural question:~can \emph{all} the structures be achieved? The main novelty that this paper provides is a positive answer to this question, in a sense that we will make precise in the course of the paper, cf., Corollary \ref{cor:existence_methods}.

	The rest of this paper is organized as follows. In Section \ref{sec:background} we introduce some preliminary notions, all the terminology and results related to the degenerate PPP framework and to the theory of FRS methods, along with the notion of bilevel graph. Section \ref{sec:general_case} presents the proposed graph-based extensions of the DRS method along with some properties. In Section \ref{sec:distributed} we show how the graph-based DRS can be leveraged to design new fully distributed schemes for \eqref{eq:Nop} assuming tree or more general base graphs. In Section \ref{sec:experiments} we show an application to a congested optimal transport problem emphasizing, in particular, an interesting influence of the algebraic connectivity of the graph topology on the convergence speed of the method. We conclude with an application to distributed Support Vector Machines, showing that the devised distributed schemes reach highly competitive performances compared to state-of-the-art methods such as P-EXTRA \cite{PG-EXTRA} and a distributed variant of the PDHG method \cite{Chambolle2011}.

	\section{Background and preliminary results}\label{sec:background}
	Let $\bH$ be a real Hilbert space, $\mathcal{A}:\bH\to 2^{\bH}$ be a maximal monotone operator and let $\mathcal{M}:\bH\to\bH$ be a self-adjoint linear bounded operator. Finding a zero of $\mathcal{A}$, i.e., a point $u \in \bH$ such that $0 \in \mathcal{A}u$, could be formulated as a fixed-point inclusion problem $u \in \mathcal{T}u$, with $\mathcal{T}:=\left(\mathcal{M}+\mathcal{A}\right)^{-1}\mathcal{M}$. Even if $\mathcal{M}$ is not invertible in the classical sense, we shall still consider $\mathcal{M}^{-1}\mathcal{A}$ as a composition of multivalued operators and it holds that $\mathcal{T}=\left(I+\mathcal{M}^{-1}\mathcal{A}\right)^{-1}$. Note that $\cM$ defines a seminorm on $\bH$, that is $\|u\|_{\cM}^2 = \langle \cM u, u\rangle$ for all $u \in \bH$. The following decomposition of $\cM$ will be useful, see \cite[Proposition 2.3]{bredies2021degenerate} for a proof.

	\begin{prop}\label{prop:onto_decomposition}
		Let $\mathcal{M}:\bH\to \bH$ be a self-adjoint, linear, bounded, positive semidefinite operator. Then, there exists an injective operator $\mathcal{C}:\bD\to \bH$, for some real Hilbert space $\bD$, such that $\mathcal{M}=\mathcal{C}\mathcal{C}^*$. Moreover, if $\mathcal{M}$ has closed range, then $\mathcal{C}^*$ is onto.
	\end{prop}

	When $\mathcal{M}$ has closed range, we call \emph{any} factorization $\mathcal{M}=\cC\cC^*$, with $\cC:\bD\to \bH$ injective and $\bD$ a Hilbert space, an \emph{onto decomposition} of $\cM$.  In the following result, we prove that, once $\bD$ is fixed, onto decompositions are unique modulo orthogonal transformations.

	\begin{prop}\label{prop:uniqueness_onto}
		Let $\mathcal{M}:\bH\to \bH$ be a self-adjoint, linear, bounded, positive semidefinite operator with closed range. Then, $\cM = \cC\cC^*$, with $\cC:\bD\to \bH$, is an onto decomposition of $\mathcal{M}$ if and only if for every onto decomposition $\cM = \widetilde{\cC}\widetilde{\mathcal{C}}^*$, with $\widetilde{\cC}:\widetilde{\bD}\to \bH$, there exists a linear isomorphism $\mathcal{O}:\bD\to \widetilde{\bD}$ with $\mathcal{O}^{-1} = \mathcal{O}^*$, such that $\mathcal{C}=\widetilde{\mathcal{C}}\mathcal{O}$.
	\end{prop}
	\begin{proof}
		First, note that, since $\cM$ has closed range, $\Img \cM$ equipped with the $\cM$-seminorm defines a Hilbert space. Further, for every onto decomposition $\cM = \cC\cC^*$, with $\cC:\bD\to \bH$, the operator $\cC^*$ is onto, and  thus, $\Img \cC = \Img \cM$. It is easy to observe that $\cC^*|_{\Img \cM}:\Img \cM\to \bD$, where $\cC^*|_{\Img \cM}$ is the restriction of $\cC^*$ to $\Img \cM=(\ker \cC^*)^\perp$, and $\cC:\bD\to \Img \cM$ define linear isomorphisms.

		Now, given two onto decompositions $\cM = \cC\cC = \widetilde{\cC}\widetilde{\cC}^*$ with $\mathcal{C}:\bD\to \bH$ and $\widetilde{\cC}:\widetilde{\bD}\to \bH$, we have $\mathcal{C}^*\mathcal{C}\mathcal{C}^*=\mathcal{C}^*\widetilde{\mathcal{C}}\widetilde{\mathcal{C}}^*$ and, since $\cC^*\mathcal{C}$ is a linear isomorphism, we can write
		\begin{equation}\label{eq:uniqueness_onto_proof}	\mathcal{C}^*=(\mathcal{C}^*\mathcal{C})^{-1}\mathcal{C}^*\widetilde{\mathcal{C}}\widetilde{\mathcal{C}}^*.
		\end{equation}
		From \eqref{eq:uniqueness_onto_proof}, it follows that
		$\cC^*\cC=(\cC^*\cC)^{-1}\cC^*\widetilde{\cC}\widetilde{\cC}^*\cC$ and again, since $\cC^*\cC$ is invertible, we get
		\begin{equation}\label{eq:uniqueness_onto_proof2}
			I=\big((\cC^*\cC)^{-1}\cC^*\widetilde{\cC}\big)\big(\widetilde{\cC}^*\cC(\cC^*\cC)^{-1}\big).
		\end{equation}
		Therefore, letting $\mathcal{O}=\widetilde{\cC}^*\cC(\cC^*\cC)^{-1}$, we get from \eqref{eq:uniqueness_onto_proof2} that $\mathcal{O}^*\mathcal{O}=I$, and from \eqref{eq:uniqueness_onto_proof}, taking adjoints, that $\cC = \widetilde{\cC}\mathcal{O}$. Being a composition of two linear isomorphisms, namely $\widetilde{\cC}^*\cC$ and $(\cC^*\cC)^{-1}$, the operator $\mathcal{O}$ is a linear isomorphism between $\bD$ and $\widetilde{\bD}$, and $\mathcal{O}^{-1}=\mathcal{O}^*$. The converse statement is immediately clear.
	\end{proof}

	In \cite{bredies2021degenerate} we show that if the preconditioner $\cM$ has closed range and $\mathcal{M}=\mathcal{C}\mathcal{C}^*$ is an onto decomposition with $\cC:\bD\to \bH$, proximal point iterations with respect to $\mathcal{M}^{-1}\mathcal{A}$ are in some sense equivalent to proximal point iterations with respect to the so-called \emph{parallel composition} $\mathcal{C}^* \rhd \mathcal{A} := \left(\mathcal{C}^* \mathcal{A}^{-1}\mathcal{C}\right)^{-1}$, which is defined on $\bD$. The reason lies in the following result, proven in \cite{bredies2021degenerate} and in \cite{parallel_composition_briceno} simultaneously.

	\begin{lemma}\label{lem:PushForward}
		Let $\mathcal{A}:\bH\to 2^{\bH}$ be an operator, $\mathcal{M}:\bH\to\bH$ be an admissible preconditioner with closed range and $\mathcal{M}=\mathcal{C}\mathcal{C}^*$ be an onto decomposition with $\cC:\bD\to \bH$. Then, the operator $\mathcal{C}^* \rhd \mathcal{A}$ is maximal monotone in $\bD$ and
		\begin{equation}\label{eq:reduced_resolvent}
			\left(I+\mathcal{C}^* \rhd \mathcal{A}\right)^{-1} = \mathcal{C}^*\left(\mathcal{M}+\mathcal{A}\right)^{-1}\mathcal{C}.
		\end{equation}
	\end{lemma}

	\paragraph{The reduced scheme.} An onto decomposition of $\mathcal{M}$ allows us to derive what we called in \cite{bredies2021degenerate} the \emph{reduced} PPP method. Indeed, since $\mathcal{M}=\mathcal{C}\mathcal{C}^*$ for $\cC:\bD\to \bH$, the general PPP iteration writes
	\begin{equation}\label{eq:from_PPP_to_reduced}
		u^{k+1} = u^k+\theta_k\left(\left(\mathcal{M}+\mathcal{A}\right)^{-1}\mathcal{C}\mathcal{C}^*u^k -u^k\right).
	\end{equation}
	Simply applying $\mathcal{C}^*$ to \eqref{eq:from_PPP_to_reduced} and considering $w^{k} = \mathcal{C}^* u^k$ for all $k \in \mathbb{N}$, we get, using \eqref{eq:reduced_resolvent}, that
	\begin{align}\label{eq:reduced_PPP}
		w^{k+1} & = w^k+\theta_k\left( \mathcal{C}^*\left(\mathcal{M}+\mathcal{A}\right)^{-1}\mathcal{C} w^k-w^k \right)\nonumber \\
		        & = w^k + \theta_k\left( \left(I+\mathcal{C}^* \rhd \mathcal{A}\right)^{-1}w^k-w^k \right).
	\end{align}
	The method \eqref{eq:reduced_PPP} is a then classical proximal point scheme with respect to the operator $\left(I+\mathcal{C}^* \rhd \mathcal{A}\right)^{-1}$ that we denote by $\widetilde{\mathcal{T}}$. We called such a method the \textit{reduced} PPP method since, as we also seen in the proof of Proposition \ref{prop:uniqueness_onto}, the space $\bD$ is (isometrically) isomorphic to the Hilbert space $\Img \mathcal{M}$ endowed with the $\mathcal{M}$-seminorm, which, if $\mathcal{M}$ is degenerate, is strictly contained in $\bH$.

	\begin{remark}\label{rem:independence_onto_dec}
		The reduced scheme does not depend on the onto decomposition of $\mathcal{M}$. Indeed, if $\mathcal{C}\mathcal{C}^*=\widetilde{\mathcal{C}}\widetilde{\mathcal{C}}^*=\mathcal{M}$ are two onto decompositions with $\cC:\bD\to \bH$ and $\widetilde{\cC}:\widetilde{\bD}\to \bH$, then by Proposition \ref{prop:uniqueness_onto} we have $\mathcal{C} =\widetilde{\mathcal{C}}\mathcal{O}$ for some linear isomorphism $\mathcal{O}$ with $\mathcal{O}^{-1} = \mathcal{O}^*$. Thus, since the corresponding reduced sequences, $\{w^k\}_k$ and $\{\widetilde{w}^k\}_k$, satisfy $w^k = \mathcal{C}^*u^k$ and $\widetilde{w}^k = \widetilde{\mathcal{C}}^*u^k$, we have, for all $k \in \mathbb{N}$,
		\begin{equation}
			w^k = \mathcal{C}^*u^k = \mathcal{O}^*\widetilde{\mathcal{C}}^*u^k =  \mathcal{O}^*\widetilde{w}^k.
		\end{equation}
		Therefore, $\mathcal{O}w^k = \widetilde{w}^k$ for all $k \in \mathbb{N}$, and, hence, the two algorithms are equivalent.
	\end{remark}

	\paragraph{Convergence result.} The convergence analysis for degenerate PPP methods has been investigated in \cite{bredies2021degenerate}. We summarize the main convergence result.

	\begin{theorem}\label{thm:convergence}
		Let $\mathcal{A}:\bH\to 2^{\bH}$ with $\zer\mathcal{A}\neq\emptyset$ be a maximal monotone operator and $\mathcal{M}$ be an admissible preconditioner with closed range. Let  $\mathcal{M}=\mathcal{C}\mathcal{C}^*$ be an onto decomposition of $\mathcal{M}$ with $\mathcal{C}:\bD\to\bH$, and let $\mathcal{T}=(I+\cM^{-1}\cA)^{-1}$, $\widetilde{\mathcal{T}} = (I+\cC^*\rhd \cA)^{-1} $. Denote by $\{u^k\}_k$ the PPP sequence according to \eqref{eq:PPP} and $\{w^k\}_k$ the corresponding reduced sequence according to \eqref{eq:reduced_PPP}. Then, we have
		\begin{enumerate}
			\item $\{w^k\}_k$ weakly converges in $\bD$ to a point $w^*\in \bD$ such that $u^* = \left(\mathcal{M}+\mathcal{A}\right)^{-1}\mathcal{C}w^*\in \zer \mathcal{A}$.
			\item If $(\mathcal{M}+\mathcal{A})^{-1}$ is Lipschitz, then $\{(\mathcal{M}+\mathcal{A})^{-1}\mathcal{C}w^k\}_k$ weakly converges to $u^*$.
		\end{enumerate}
	\end{theorem}

	\noindent We refer to \cite[Theorem 2.14]{bredies2021degenerate} and \cite[Corollary 2.15]{bredies2021degenerate} for a proof and further comments.

	\subsection{Frugal resolvent splitting methods}\label{sec:frsm}
	To tackle the $N$-operator problem \eqref{eq:Nop} we consider the class of FRS methods introduced by Ryu in \cite{Ryu} and, later, further investigated in \cite{malitsky2021resolvent}. Here, for reader's convenience, we outline the main properties and results on this class of methods, referring for further details to \cite{malitsky2021resolvent}. Let $H$ be a Hilbert space and for $N\geq 1$ let $\MM_{N}$ be the set of all $N$-tuples of maximal monotone operators on $H$.

	\begin{defi}[Fixed-point encoding]
		Let $\bD$ and $H$ be Hilbert spaces, a pair of operators $(T, S)$, with $T: \MM_{N}\times \bD \to \bD$ and $S: \MM_{N}\times \bD \to H$, is a fixed-point encoding for $\MM_{N}$ if, for all $\mathbf{A}=(A_1, \dots, A_N) \in \MM_{N}$, the following hold
		\begin{enumerate}
			\item $\Fix T(\mathbf{A}, \cdot) \neq \emptyset$ if and only if $
				      \zer \left(A_1+\cdots + A_N\right) \neq \emptyset,$
			\item If $\bw = T(\mathbf{A}, \bw)$, then $S(\mathbf{A}, \bw) \in \zer \left(A_1+\cdots + A_N\right)$.
		\end{enumerate}
		The map $T$ is called \textit{fixed-point operator}, and $S$ is called \textit{solution operator}.
	\end{defi}
	Fixed-point encodings, and in particular fixed-point operators, naturally define an associated fixed-point algorithm, namely for any $\mathbf{A}\in\MM_N$,
	\begin{equation}\label{eq:fixed_point_algorithm}
		w^{k+1}=T(\mathbf{A},w^k), \quad \text{for} \ w^0 \in \bD.
	\end{equation}
	\begin{defi}[Unconditional stability]
		The fixed-point encoding $(T,S)$ is \emph{unconditionally stable} if for any starting point $w^0\in \bD$ and any $\mathbf{A}=(A_1,\dots, A_N)\in\MM_N$ with $\zer (A_1+\cdots+A_N)\neq \emptyset$, the corresponding fixed-point algorithm \eqref{eq:fixed_point_algorithm} weakly converges to a fixed-point of $T(\mathbf{A},\cdot)$.
	\end{defi}
	For the case $N=1$ we have for instance: $\mathbf{A} = (A_1)$, $T(\mathbf{A}, \cdot) = J_{A_1}$ and $S(\mathbf{A}, \cdot) = I$, where $I$ is the identity operator on $H$, which corresponds to the proximal point algorithm. For $N=2$, we can choose: $\mathbf{A} = (A_1, A_2)$, $T(\mathbf{A}, \cdot) = I+J_{A_2}\left(2J_{A_1}-I\right)-J_{A_1}$ and $S(\mathbf{A}, \cdot) = J_{A_1}$, which yields the Douglas--Rachford algorithm. We also notice that in the two examples above the operators $T$ and $S$ can be evaluated efficiently applying successively (and only once) $J_{A_1}$ and $J_{A_2}$, which are assumed to be simple enough. This idea can be fixed by a definition.
	\begin{defi}[Frugal resolvent splitting] We say that a fixed-point encoding $(T, S)$ is a resolvent splitting if, for all $\mathbf{A}=(A_1, \dots, A_N) \in \MM_{N}$, there is a finite procedure that evaluates $T(\mathbf{A}, \cdot)$ and $S(\mathbf{A}, \cdot)$ at a given point that uses only vector addition, scalar multiplication, and the resolvents of $A_1, \dots, A_N$. A resolvent splitting is \textit{frugal} if, in addition, each of the resolvents of $A_1, \dots, A_N$ is evaluated exactly once.
	\end{defi}
	Given $\mathbf{A}=(A_1,\dots, A_N)\in \MM_N$ with $\zer(A_1+\cdots+A_N)\neq \emptyset$, a FRS method is the fixed-point algorithm associated with a frugal resolvent splitting $(T,S)$, with $T(\mathbf{A}, \cdot):\bD\to \bD$. Note that, roughly speaking, if the space $\bD$ is \textit{large}, implementing a FRS method may lead to huge memory requirements. For this reason, one should put adequate care on the definition of $\bD$.
	\begin{defi}[Lifting]\label{def:lifting}
		Let $d\in \mathbb{N}$. A fixed-point encoding $(T, S)$ has a $d$-fold lifting for $\MM_{N}$ if $\bD = H^d$.
	\end{defi}
	A ground-breaking series of results initiated by Ryu in \cite{Ryu} for the three-operator problem and later extended by Malitsky and Tam in \cite{malitsky2021resolvent} for the general problem states that there is an inherent \emph{lower} bound on the number of variables, i.e.~$d$, for an unconditionally stable FRS method.
	\begin{theorem}[Minimal lifting \cite{malitsky2021resolvent, Ryu}] Let $(T, S)$ be an unconditionally stable FRS for $\MM_{N}$ with a $d$-fold lifting. If $N \geq 2$, then $d\geq N-1$.
	\end{theorem}
	The authors proceed to show that the bound $N-1$ is tight, and it is for this reason that we say that an unconditionally stable FRS method for $\MM_N$ has \emph{minimal variables} or a \emph{minimal lifting} when $\bD=H^{N-1}$, i.e., algorithm \eqref{eq:fixed_point_algorithm} requires storing $N-1$ variables living in $H$. Interestingly, in \cite{malitsky2021resolvent} the authors also provided an explicit characterization of the general structure of a FRS, which will be helpful for our subsequent discussion.
	\begin{lemma}[Lemma 3.1 in \cite{malitsky2021resolvent}]\label{lem:frs_char} Let $(T, S)$ be a FRS for $\MM_{N}$ with a $d$-fold lifting. Let $I$ be the identity on $H$ and fix $\mathbf{A}=(A_1, \dots, A_N)\in\MM_{N}$. Then, for all $\bw=(w_1,\dots,w_d)\in H^d$:
		\begin{equation*}
			T(\mathbf{A}, \bw) = (T_w \otimes I) \bw + (T_x \otimes I) \bx,
		\end{equation*}
		where $T_w\in \mathbb{R}^{d\times d}, \ T_x \in \mathbb{R}^{d\times N}$, and $\bx=(x_1, \dots, x_N) \in H^{N}$ is given by
		\begin{equation*}
			x_i = J_{\sigma_i A_i}\bigg(\sum_{h=1}^i l_{hi}x_h+\sum_{j=1}^{N-1} b_{ij}w_j\bigg),
		\end{equation*}
		where $\sigma_i > 0$ for all $i \in \{1, \dots, N\}$, $(l_{hi})_{hi}$ are the components of a (strictly) lower triangular matrix $L \in \mathbb{R}^{N\times N}$ and $(b_{ij})_{ij}=B \in \mathbb{R}^{N\times d}$.
	\end{lemma}

	\subsection{State and bilevel graphs}

	A directed graph is a pair $\mathcal{G} = (\mathcal{N}, \mathcal{E})$, where $\mathcal{N}$ is a finite set and $\mathcal{E}$ a subset of $\mathcal{N}\times \mathcal{N}$. The elements of $\mathcal{N}$ are the nodes of the graph, the elements of $\mathcal{E}$ its edges. Two nodes $i$ and $j$ are \textit{adjacent} if $(i, j) \in \mathcal{E}$ or $(j,i) \in \cE$. We denote the set of adjacent nodes to $i$ in $G$ by $\text{adj}(i; G)$. The \emph{degree} of a node $i$ is the cardinality of $\text{adj}(i; G)$ and we often denote it by $d_i$. A \emph{path} between $i_0$ and $i_n$ is a sequence of distinct nodes $(i_0, i_1, \dots, i_n)$ with $i_k$ and $i_{k+1}$ adjacent for $k = 0, \dots, n-1$. Two nodes of $\mathcal{N}$ are connected if there exists at least one path that has the two nodes as its end points. A graph is \emph{connected} if every pair of nodes is connected.

	For a directed graph $G=(\cN, \cE)$ an ordering of $\cN$ is a bijection $\alpha : \{1,\dots, N\}\to \cN$. The triple $(\cN, \cE, \alpha)$ is sometimes referred to as ordered graph. In the remainder of this paper, we will refer to the couple $G =(\cN, \cE)$ as a directed ordered graph via the identification $\cN = \{1,\dots, N\}$. A \emph{topological ordering} of $G$ is an ordering such that if $(i,j)\in \cE$, then $i<j$. For an ordered directed graph, the \emph{in-degree} (resp. the \emph{out-degree}) of $i$ is the cardinality of $\text{adj}(i, G)\cap \{ h \mid h<i \}$ (resp. the cardinality of $\text{adj}(i, G)\cap \{ h \mid i<h \}$) and is denoted by $d_i^+$ (resp. $d_i^-$).

	Recall from Lemma \ref{lem:frs_char} that each FRS can be characterized by means of four matrices $T_w$, $T_x$, $L$, $B$ and a vector $\sigma=(\sigma_1, \dots, \sigma_N)$, where $L$ is strictly lower triangular, i.e., with zero diagonal and upper triangular part. The structure of $L$ imposes a topological order on the evaluations of the resolvents of $A_1, \dots, A_N$. We can therefore associate to each FRS a directed graph with a topological ordering.

	\begin{defi}[State graph of a FRS]
		Let $(T, S)$ be a FRS for $\MM_N$ and let $L$ be the triangular matrix given by Lemma \ref{lem:frs_char}. The state graph associated with $(T, S)$ is an ordered directed graph $G=(\cN, \cE)$ with $\cN = \{1,\dots, N\}$ and $(i,j)\in \cE$ if and only if $l_{ij}\neq 0$. A FRS method has state graph $G$ if the corresponding FRS $(T,S)$ has state graph $G$.
	\end{defi}

	Our construction of the graph extension of the DRS method requires endowing the notion of state graph with an additional layer.

	\begin{defi}[Bilevel graphs]\label{def:bilevel_graph}
		Let $N\in \mathbb{N}$. A bilevel graph is a triple $biG = (\mathcal{N}, \mathcal{E}, \mathcal{E}')$ where $G=(\mathcal{N}, \mathcal{E})$ is a connected directed graph with a topological ordering $\cN = \{1, \dots, N\}$ and $G'=(\mathcal{N}, \mathcal{E}')$ is a directed connected subgraph. We call $G$ the state graph and $G'$ the base graph.
	\end{defi}
	In Section \ref{sec:general_case}, we show that for any bilevel graph $biG=(\cN, \cE, \cE')$ there exists an unconditionally stable FRS method for $\MM_N$ with a minimal lifting and state graph $G=(\cN, \cE)$. The choice of the base graph further discriminates the resulting schemes and is related to the operators $B$ and $T_x$ from Lemma \ref{lem:frs_char}.

	\section{Graph-based Douglas--Rachford}\label{sec:general_case}

	In our construction, the Laplacian of the base graph plays a key role.
	\begin{defi}
		[Graph Laplacian] Given a directed graph $G = (\mathcal{N}, \mathcal{E})$, with $\mathcal{N}=\{1, \dots, N\}$, the graph Laplacian of $G$ is the matrix $L=(L_{ij})_{ij} \in \mathbb{R}^{N\times N}$ defined by
		\begin{equation*}
			L_{ij}:=
			\begin{cases}
				d_i & \text{if} \ i=j,                    \\
				-1  & \text{if $i$ and $j$ are adjacent}, \\
				0   & \text{else}.
			\end{cases}
		\end{equation*}
		where $d_1, \dots, d_N \in \mathbb{R}$ are the degrees of the nodes $1, \dots, N$, respectively.
	\end{defi}

	\begin{lemma}\label{lem:onto_laplacian}
		Let $N>1$. For each connected graph $G=(\mathcal{N}, \mathcal{E})$ with $\cN = \{1,\dots, N\}$ there exists $z_1, \dots, z_N\in \mathbb{R}^{N-1}$ such that
		\begin{enumerate}
			\item[\textit{(a)}] $z_i \cdot z_j \neq 0$ if and only if $i$ and $j$ are adjacent;
			\item[\textit{(b)}] It holds $z_1+\dots +z_N = 0$;
			\item[\textit{(c)}] $\Span\{z_1, \dots, z_N\}=\mathbb{R}^{N-1}$.
		\end{enumerate}
	\end{lemma}
	\begin{proof}
		Let $L=(L_{ij})_{ij}\in \mathbb{R}^{N\times N}$ be the Laplacian of the graph $G$. The matrix $L$ has the following properties: $L$ is symmetric and positive semidefinite; $L$, since $G$ is connected, has rank $N-1$; for $1\leq i\leq j \leq N$, $(i,j) \in \mathcal{E}$ if and only if $L_{ij} \neq 0$, and $\mathbf{1} \in \ker L$, where $\mathbf{1}=(1,\dots, 1)^*$, see, e.g., \cite{DEABREU200753}. Let $L = ZZ^*$ be an onto decomposition of $L$, where $Z\in \mathbb{R}^{N\times (N-1)}$ and set $z_1, \dots, z_N \in \mathbb{R}^{N-1}$ to be the rows $Z$. Using that $\ker Z^* = \ker L$ and the fact that $z_i\cdot z_j = L_{ij}$ it is clear that $z_1, \dots, z_N$ satisfy \textit{(a), (b)} and \textit{(c)}.
	\end{proof}

	A collection $z_1,\dots, z_N\in \mathbb{R}^{N-1}$ that satisfies \textit{(a)} and \textit{(c)} in Lemma \ref{lem:onto_laplacian} is often called a \emph{faithful orthogonal representation} of $G$, see \cite{LOVASZ1989439}. Here, we seek for a faithful orthogonal representation that also sums up to zero.

	\begin{remark}[Gossip matrices]
		The matrix $Z\in \mathbb{R}^{N\times (N-1)}$ introduced in the proof of Lemma \ref{lem:onto_laplacian}, with rows $z_i$, has full rank and $\ker Z^*=\Span\{\mathbf{1}\}$. In practice, $Z$ can be obtained deriving an onto decomposition of the Laplacian matrix of $G$, which can be done via spectral decomposition. If $G$ is a tree, one can take the incidence matrix (cf., Section \ref{sec:tree_graphs}). Note that, in general, one could replace the Laplacian with any symmetric positive semidefinite matrix $W=(W_{ij})_{ij}$ of rank $N-1$ such that for $1\leq i\leq j \leq N$, $(i,j) \in \mathcal{E}$ if and only if $W_{ij} \neq 0$, and $\mathbf{1} \in \ker W$, where $\mathbf{1}=(1,\dots, 1)^*$. This class of operators is often called \textit{Gossip matrices} \cite{EXTRA}. In the remainder of this paper, we stick to the Laplacian for simplicity.
	\end{remark}

	Let $G$ be a state graph for the $N$-operator problem \eqref{eq:Nop}. In order to find an unconditionally stable FRS method with minimal (i.e., $N-1$) variables and associated state graph $G$, we stick to the following methodology. We find a maximal monotone operator $\cA:\bH\to 2^{\bH}$ on the Hilbert space $\bH:=H^{2N-1}$ such that if $0\in \cA u$ then the first $N$ components of $u$ are equal and solve \eqref{eq:Nop} and, conversely, if $x \in H$ solves \eqref{eq:Nop} there exists $u \in \bH$ such that $0\in \mathcal{A}u$ and the first $N$ components of $u$ all coincide with $x$. Then, we design an admissible preconditioner $\mathcal{M}:\bH\to \bH$ for $\cA$ that admits an onto decomposition with $\bD := H^{N-1}$. In this way, the corresponding reduced PPP method according to \eqref{eq:reduced_PPP} would need to store exactly $N-1$ variables. Imposing certain structural properties to $\cA$ and $\cM$ the method will also meet the desired structure.

	\paragraph{Building $\mathcal{M}=\mathcal{C}\mathcal{C}^*$.} Consider the base graph $G'$. Let $Z=(Z_{ij})_{ij}\in \mathbb{R}^{N \times (N-1)}$ be a matrix whose rows are the vectors given by Lemma \ref{lem:onto_laplacian} applied to $G'$. Recall that, in particular, we can choose $Z$ such that $L=ZZ^*$ where $L$ is the graph Laplacian of $G'$. Let $\mathcal{C}$ be the following operator
	\begin{equation}\label{eq:def_C}
		\mathcal{C}^* =
		\begin{bmatrix}
			\mathcal{Z}^* & \mathcal{I}
		\end{bmatrix},
	\end{equation}
	where $\mathcal{Z}^*=Z^* \otimes I$, $\mathcal{I}$ is the identity map on $H^N$, and $I$ is the identity map on $H$. Once the operator $\mathcal{C}:\bD\to \bH$ is fixed, the preconditioner can be obtained as $\mathcal{M}=\mathcal{C}\mathcal{C}^*$, which yields
	\begin{equation}\label{eq:M_block}
		\mathcal{M}=
		\begin{bmatrix}
			\mathcal{L}   & \mathcal{Z} \\
			\mathcal{Z}^* & \mathcal{I}
		\end{bmatrix},
	\end{equation}
	where $\mathcal{L}=\cZ \cZ^*=L\otimes I$. Note that, since $\cC^*$ is onto, the factorization $\cM=\cC\cC^*$ is an onto decomposition of $\cM$ with $\cC:\bD\to \bH$.

	\paragraph{Building $\mathcal{A}$.}  Let $\mathbf{A}\in \MM_N$. We need to find a suitable maximal monotone operator $\mathcal{A}$ on $\bH$ such that the reduced PPP method with respect to $\cA$ and $\cM=\cC\cC^*$ defined in \eqref{eq:M_block} meets the desired structure. We first define
	\begin{equation*}
		\Sigma =
		\begin{bmatrix}
			0      & -L_{12} & \cdots                    & \hspace{-0.4cm}-L_{1N}    \\
			L_{21} & 0       &                           & \hspace{-0.4cm}\vdots     \\
			\vdots &         & \ddots                    & \hspace{-0.4cm}-L_{N-1,N} \\
			L_{N1} & \cdots  & \hspace{-0.4cm} L_{N,N-1} & \hspace{-0.4cm} 0
		\end{bmatrix}
	\end{equation*}
	where $L_{ij}$ are the components of the graph Laplacian of $G'$. We denote by $\boldsymbol{A}$ the diagonal operator $\boldsymbol{A}: (x_1, \dots, x_N)\mapsto (A_1x_1, \dots, A_Nx_N)$. Then, we set $\mathcal{B}_L := \boldsymbol{A} + \boldsymbol{\Sigma}$, where $\boldsymbol{\Sigma} =\Sigma \otimes I$. Now, consider the difference $\cE\setminus\cE':=\{(i,j) \in \cE \mid (i,j) \not \in \cE'\}$ and let $\mathcal{P}: H^{N}\to H^N$ be the operator defined as $\mathcal{P}=\sum_{(i,j)\in \mathcal{E}\setminus\cE'}\mathcal{P}^{ij}$ where, for each $(i,j) \in \mathcal{E}\setminus\cE'$, the operator $\mathcal{P}^{ij}$ is given by $\mathcal{P}^{ij}:=P^{ij}\otimes I$ with
	\begin{equation*}
		P^{ij}\in \mathbb{R}^{N\times N} \ \text{defined by:} \
		(P^{ij})_{hk} :=
		\begin{cases}
			1  & \text{if $h=i$ and $k = i$}, \\
			1  & \text{if $h=j$ and $k = j$}, \\
			-2 & \text{if $h=j$ and $k = i$}, \\
			0  & \text{else}.
		\end{cases}
	\end{equation*}
	Eventually, we define $\mathcal{A}_L := \mathcal{B}_L+\mathcal{P}$ and build the operator $\mathcal{A}$ on $\bH$ assembling $\mathcal{A}_L$ and $\cZ$ in the following block structure
	\begin{equation}\label{eq:A_block}
		\mathcal{A} :=
		\begin{bmatrix}
			\mathcal{A}_L & \hspace{-0.2cm}-\mathcal{Z} \\
			\mathcal{Z}^* & \boldsymbol{0}
		\end{bmatrix},
	\end{equation}
	where $\boldsymbol{0}$ is the zero operator on $\bD$.
	\begin{theorem}\label{thm:over_lifting_max_mon}
		Let $biG = (\cN, \cE, \cE')$ be a bilevel graph for the $N$-operator problem \eqref{eq:Nop} with respect to $\mathbf{A}\in \MM_N$ with $\zer (A_1+\dots+A_N)\neq \emptyset$. Let $\mathcal{A}:\bH\to 2^{\bH}$ be the operator defined in \eqref{eq:A_block}. Then, $\mathcal{A}$ is maximal monotone, and for $(x_1,\dots, x_N)\in H^N$, there exists $(v_1,\dots, v_{N-1})\in \bD$ such that $u = (x_1,\dots, x_N, v_1, \dots, v_{N-1})\in \bH$ satisfies $0 \in \mathcal{A}u$ if and only if $x = x_1 = \dots = x_N\in H$ solve \eqref{eq:Nop}.
		\begin{proof}
			First, we suppose that $u = (x_1, \dots, x_N, v_1,\dots v_{N-1})$ is such that $0 \in \mathcal{A}u$. Let us denote $\bx=(x_1, \dots, x_N)$ and $\bv=(v_1,\dots v_{N-1})$. By construction, we have that $\mathcal{Z}^*\bx=0$, which implies $x_1 = \dots = x_N=x$ (by definition and Lemma \ref{lem:onto_laplacian}\textit{(b)}). Now, the first block-row of \eqref{eq:A_block} yields
			\begin{equation*}
				0 \in \mathcal{A}_L\bx -\cZ \bv = \mathcal{B}_L \bx + \mathcal{P}\bx -\cZ\bv = \boldsymbol{A}\bx +\boldsymbol{\Sigma} \bx + \mathcal{P}\bx-\cZ\bv.
			\end{equation*}
			Thus, there exists $\boldsymbol{a} = (a_1, \dots, a_N)$ with $a_i \in A_ix$ for all $i \in \{1, \dots, N\}$ such that
			\begin{equation}\label{eq:max_mon_proof}
				\boldsymbol{0} = \boldsymbol{a} +\boldsymbol{\Sigma} \bx + \mathcal{P}\bx-\cZ\bv.
			\end{equation}
			Note that $\bx = (\mathbf{1}\otimes I)x$ and, thus, applying $(\mathbf{1}\otimes I)^* =(\mathbf{1}^*\otimes I)$ to $\boldsymbol{\Sigma} \bx$, as $\boldsymbol{\Sigma}$ is skew-symmetric, yields $0\in H$. Recall that $\mathcal{P} = P \otimes I$, with $P = \sum_{(i,j) \in \cE\setminus\cE'}P^{ij}$. It is clear that $\boldsymbol{1}^*P^{ij}\boldsymbol{1} = 0$ for each $(i,j) \in \cE\setminus\cE'$, hence $(\boldsymbol{1}\otimes I)^*\mathcal{P}\bx = 0 \in H$. Eventually, also $(\boldsymbol{1}\otimes I)^*\mathcal{Z}\bv=(\boldsymbol{1}^*Z\otimes I) \bv=0$, as $\ker Z^* = \Span\{\mathbf{1}\}$. In summary, applying $(\boldsymbol{1}\otimes I)^*$ to  \eqref{eq:max_mon_proof} we get
			\begin{equation*}
				0 = (\boldsymbol{1}\otimes I)^*\boldsymbol{a} =\sum_{i=1}^Na_i \in \sum_{i=1}^NA_ix.
			\end{equation*}
			On the other hand, if we have a solution $x$ of \eqref{eq:Nop}, i.e., there exist $a_i \in A_ix$ for all $i \in \{1,\dots, N\}$ such that $\sum_{i=1}^Na_i=0$, we define $\bx=(x_1,\dots,x_N)$ with $x_1 = \dots = x_N=x$ and $\boldsymbol{a}=(a_1, \dots, a_N)$. In this way, $\cZ^*\bx = 0$. To conclude, we only need to find $\bv= (v_1,\dots v_{N-1})$ such that \eqref{eq:max_mon_proof} holds. Such an element can be found as a solution of the linear system $\mathcal{Z}\bv = \boldsymbol{a} +\Sigma \bx + \mathcal{P}\bx$, which always exists. Indeed, since $\bx = (\mathbf{1}\otimes I)x$,  $\boldsymbol{\Sigma} = \Sigma \otimes I$ and $\mathcal{P} = P\otimes I$ with $\mathbf{1}^* \Sigma \mathbf{1} = \mathbf{1}^* P \mathbf{1} = 0$ (because $\Sigma$ is skew-symmetric and $P = \sum_{(i,j) \in \cE\setminus\cE'}P^{ij}$ with $\mathbf{1}^*P^{ij}\mathbf{1}=0$ for all $(i,j) \in \cE\setminus \cE'$), we have $(\mathbf{1}\otimes I)^*(\boldsymbol{a}+(\boldsymbol{\Sigma}+\mathcal{P})\bx) = \sum_{i=1}^N a_i = 0$, i.e., the right-hand side obeys $\boldsymbol{a} +\boldsymbol{\Sigma} \bx + \mathcal{P}\bx \in (\ker \cZ^* )^{\perp}= \Img \mathcal{Z}$.

			For the maximal monotonicity of $\mathcal{A}$, let us first note that since $\zer(A_1+\dots+A_N)\neq \emptyset$, $\dom \cA\neq \emptyset$. Recall from \eqref{eq:A_block} that we have
			\begin{equation}\label{eq:A_block_early}
				\mathcal{A} :=
				\begin{bmatrix}
					\mathcal{B}_L & \hspace{-0.2cm}-\mathcal{Z} \\
					\mathcal{Z}^* & \boldsymbol{0}
				\end{bmatrix}
				+
				\begin{bmatrix}
					\mathcal{P}    & \boldsymbol{0} \\
					\boldsymbol{0} & \boldsymbol{0}
				\end{bmatrix},
			\end{equation}
			where the zeros may be different but are denoted the same. In \eqref{eq:A_block_early}, $\mathcal{B}_L$ is maximal monotone being the sum of a maximal monotone operator $\boldsymbol{A}$ and a skew-symmetric linear map $\boldsymbol{\Sigma}$, see \cite[Corollary 24.4]{BCombettes}. The same reasoning applies to the first term in \eqref{eq:A_block_early}. Regarding the second term in \eqref{eq:A_block_early}, we only need to show that $\mathcal{P}$ is monotone. Indeed, monotone linear maps are also maximal \cite[Example 20.15]{BCombettes}. Recall that $\mathcal{P}=\sum_{(i,j)\in \cE\setminus\cE'}\mathcal{P}^{ij}$ with $\mathcal{P}^{ij} = P^{ij}\otimes I$. The claim follows from the fact that the operator $P^{ij}$ is monotone for all $(i,j) \in \cE\setminus\cE'$, indeed, we can easily see that
			\begin{equation*}
				\langle P^{ij}\xi, \xi\rangle = |\xi_i-\xi_j|^2 \quad \text{for all} \ \xi = (\xi_1, \dots, \xi_N)\in H^{N}.
			\end{equation*}
			The maximality of $\mathcal{A}$ follows for instance from \cite[Corollary 24.4]{BCombettes}.
		\end{proof}
	\end{theorem}

	\paragraph{General iterations.} We can derive a closed-form expression for the reduced PPP method derived with respect to the maximal monotone operator $\mathcal{A}$ and the preconditioner $\mathcal{M}=\cC\cC^*$ given by \eqref{eq:A_block} and \eqref{eq:M_block}, respectively. Indeed, $\mathcal{M}+\mathcal{A}$ has a lower triangular structure and, thus, one can easily derive the PPP iteration according to \eqref{eq:PPP} with $\theta_k = 1$ for all $k \in \mathbb{N}$, which read
	\begin{equation*}
		\left\{
		\begin{aligned}
			\bx^{k+1} & = \left(\mathcal{L} +\mathcal{A}_L\right)^{-1}\mathcal{Z}(\mathcal{Z}^*\bx^k+\bv^k), \\
			\bv^{k+1} & = \mathcal{Z}^*\bx^k+\bv^k-2\mathcal{Z}^*\bx^{k+1}.                                  \\
		\end{aligned}
		\right.
	\end{equation*}
	with $u^k=(\bx^k, \bv^k)\in \bH$, $\bx^k\in H^{N}$ and $\bv^k \in \bD$. The onto decomposition $\mathcal{M}=\mathcal{C}\mathcal{C}^*$ yields a reduced algorithm according to \eqref{eq:reduced_PPP} with the substitution $\bw^k = \mathcal{C}^*u^k=\mathcal{Z}^*\bx^k+\bv^k$, resulting in
	\begin{equation}\label{eq:proposed_reduced}
		\left\{
		\begin{aligned}
			\bx^{k+1} & = \left(\mathcal{L}+\mathcal{A}_L\right)^{-1}\mathcal{Z}\bw^k, \\
			\bw^{k+1} & = \bw^k-\mathcal{Z}^*\bx^{k+1}.                                \\
		\end{aligned}
		\right.
	\end{equation}
	Recall that, by construction $\bw^{k+1} = \widetilde{\mathcal{T}}\bw^k$, with $\widetilde{\mathcal{T}}:=(I+\cC\rhd\cA)^{-1}$. Thus, for general relaxation parameters $\theta_k\in (0, 2]$ such that $\sum_k \theta_k (2-\theta_k)=+\infty$, we would simply have
	\begin{equation}\label{eq:insert_relaxation}
		\bw^{k+1}=\bw^k+\theta_k\left(\widetilde{\mathcal{T}}\bw^k-\bw^k\right)=\bw^k+\theta_k\left(\bw^k-\mathcal{Z}^*\bx^{k+1}-\bw^k\right)= \bw^k-\theta_k\mathcal{Z}^*\bx^{k+1},
	\end{equation}
	where, still, $\bx^{k+1}=\left(\mathcal{L}+\mathcal{A}_L\right)^{-1}\mathcal{Z}\bw^k$. Note that \eqref{eq:insert_relaxation} consists only in a simple modification to \eqref{eq:proposed_reduced}, and that, whenever $\theta_k\neq 1$, one should not confuse $\bx^{k+1}$ with the first $N$ components of $u^{k+1}$ according to \eqref{eq:PPP}.

	\smallskip
	The operator $\left(\mathcal{L}+\mathcal{A}_L\right)$ has a lower triangular structure and thus, it is easy to invert explicitly. Indeed, for $i \in \{1, \dots, N\}$, we have
	\begin{equation*}
		\bigg[\sum_{h=1}^{i-1} 2 L_{ih} x_h^{k+1} \bigg]- \bigg[\sum_{(h, i) \in \cE\setminus\cE'}2x_h^{k+1}\bigg]+(d_i'+\bar{d_i}) x_i^{k+1} +  A_ix_i^{k+1} \ni \sum_{h=1}^{N-1} Z_{ih} w_h^k,
	\end{equation*}
	where $d_i'$ is the degree of $i$ in the base graph $G' = (\cN, \cE')$ and $\bar{d_i}$ the degree of $i$ in the graph $(\cN, \cE\setminus\cE')$. Thus, $\bar{d_i}+d_i' = d_i$, i.e., the degree of $i$ in the state graph $G$. Therefore, using that $L_{ih} = -1$ if and only if $(h, i)\in \cE'$, we get
	\begin{equation*}
		-\bigg[\sum_{(h, i) \in \cE'} 2 x_h^{k+1}\bigg] - \bigg[\sum_{(h, i) \in \cE\setminus\cE'}2x_h^{k+1}\bigg] + d_i x_i^{k+1} + A_ix_i^{k+1} \ni \sum_{h=1}^{N-1} Z_{ih} w_h^k,
	\end{equation*}
	such that we can invert explicitly provided that $A_i$ is maximal monotone for every $i\in \{1,\dots, N\}$. Further, we can insert positive step-sizes $\sigma>0$ considering for every $i\in \{1,\dots, N\}$ the operator $\sigma A_i$ instead of $A_i$. Eventually, consider a bilevel graph $biG=(\mathcal{N}, \mathcal{E}, \cE')$, an onto decomposition $L = ZZ^*$ of the graph Laplacian of the base graph $G'=(\cN, \cE')$, a step size $\sigma >0$ and relaxation parameters $\theta_k\in (0,2]$ such that $\sum_k\theta_k(2-\theta_k)=+\infty$. Denoting by $d_1,\dots, d_N$ the degrees of the nodes in the state graph, we have the following FRS method with minimal variables.

	\begin{algorithm}[H]\label{alg:splitting_1}
		\textbf{Initialize:} $w_1^0, \dots, w_{N-1}^0 \in H$\\
		\For{$k=0, 1, \dots$}{
			\For{$i=1,\dots, N$}{

				\begin{equation*}
					x_i^{k+1} = J_{\frac{\sigma}{d_i}A_i}\bigg(\frac{2}{d_i}\sum_{(h,i)\in \cE}x_h^{k+1} + \frac{1}{d_i}\sum_{j=1}^{N-1} Z_{ij} w_j^k\bigg)
				\end{equation*}
			}
			\For{$j = 1, \dots, N-1$}{
				\begin{equation*}
					w_j^{k+1} = w_j^k - \theta_k \sum_{i = 1}^{N}Z_{ij}x_i^{k+1}
				\end{equation*}
			}
		}
		\caption{The graph-based Douglas--Rachford method associated to the bilevel graph $biG$.\vspace{0.1cm}}
	\end{algorithm}

	Interestingly, the choice of different base graphs leads to different methods, where the difference can be clearly seen in the update formula for the $w$ variables. The base graph will play a crucial role in the application to distributed optimization in Section \ref{sec:distributed}.

	Following from the general framework on degenerate PPP algorithms, we can easily establish convergence.

	\begin{theorem}\label{thm:convergence_graph_drs}
		Let $biG = (\cN, \cE, \cE')$ be a bilevel graph for \eqref{eq:Nop}. Let $w_1^k, \dots, w_{N-1}^k$ and $x_1^k, \dots, x_N^{k}$ be given by Algorithm \ref{alg:splitting_1} with respect to $(A_1,\dots, A_N)\in \MM_N$ with $\zer (A_1+\cdots+A_N)\neq \emptyset$. Then, for all $i\in \{1,\dots, N-1\}$ each variable $w^k_i$ converges weakly to some $w_i^*$ such that $x_1^*, \dots, x_N^*\in H$ defined by
		\begin{equation*}
			x_i^{*} = J_{\frac{\sigma}{d_i}A_i}\bigg(\frac{2}{d_i}\sum_{(h, i)\in \mathcal{E}}x_h^{*} + \frac{1}{d_i}\sum_{j=1}^{N-1} Z_{ij}w_j^*\bigg)
		\end{equation*}
		coincide for all $i \in \{1, \dots, N\}$ and solve \eqref{eq:Nop}. Moreover, all the sequences $\{x_i^{k}\}_k$ for $i \in \{1,\dots, N\}$ converge weakly to that solution.
	\end{theorem}
	\begin{proof}
		The proof is an application of Theorem \ref{thm:convergence}. Indeed, Algorithm \ref{alg:splitting_1} is a reduced PPP method with respect to the operators $\mathcal{A}$ and $\mathcal{M}$ defined in \eqref{eq:A_block} and \eqref{eq:M_block} respectively and the onto decomposition $\cM = \cC\cC^*$ with $\cC:\bD\to \bH$ defined in \eqref{eq:def_C}. Furthermore,  $\mathcal{A}$ is maximal monotone with $\zer \cA \neq \emptyset$ from Theorem \ref{thm:over_lifting_max_mon} and $\Img\cM$ is closed by construction. The operator $(\mathcal{M}+\mathcal{A})^{-1}$ is a combination of resolvents and simple algebraic operations and is thus Lipschitz. Recall from \eqref{eq:proposed_reduced} that, by construction, $\bx^{k+1}=\left(\mathcal{L}+\mathcal{A}_L\right)^{-1}\mathcal{Z}\bw^k$ contains the first $N$ block-components of $\mathcal{T}u^k$, where $\{u^k\}_k$ is the corresponding PPP sequence, and that, since $\cC^*u^k = \bw^k$ for every $k\in \mathbb{N}$ (cf., \eqref{eq:reduced_PPP}), $\mathcal{T} u^k=\left(\cM+\cA\right)^{-1}\cC \bw^{k}$ for all $k \in \mathbb{N}$. Thus, from part 2.~of Theorem \ref{thm:convergence}, we have   $\bx^{k+1}\rightharpoonup \bx^*$, with $(\bx^*, \bv^*)\in \zer \cA$ for some $\bv^*\in \bD$. Note as well that part 1.~of Theorem \ref{thm:convergence} yields that all the sequences $\{w_i^k\}_k$ converge weakly to some elements $w_i^*$ for all $i \in \{1, \dots, N-1\}$, and, denoting by $\bw^* = (w_1^*, \dots, w_{N-1}^*)$, we also have $\bx^* = \left(\mathcal{L} +\mathcal{A}_L\right)^{-1}\mathcal{Z}\bw^*$. The claim follows applying again Theorem \ref{thm:over_lifting_max_mon}.
	\end{proof}
	\begin{corollary}\label{cor:existence_methods}
		Let $G=(\cN, \cE)$ be a directed connected graph with a topological ordering on $N$ nodes. Then, there exists an unconditionally stable FRS method with minimal variables and state graph $G$.
	\end{corollary}
	\begin{proof}
		Let $\cE' \subset \cE$ and consider the bilevel graph $biG = (\cN, \cE, \cE')$. The graph-based DRS with bilevel graph $biG$ and relaxation parameters $\theta_k=\theta\in (0,2]$ for all $k\in \mathbb{N}$ is a FRS with respect to the fixed-point encoding $(T,S)$, with $T(\mathbf{A}, \cdot)=I+\theta(\widetilde{\mathcal{T}}-I)$ and $S(\mathbf{A}, \cdot )=\pi_i(\mathcal{L}+\cA_L)^{-1}\cZ$, where $\widetilde{\mathcal{T}}=(I+\cC^*\rhd \cA)^{-1}$ and $\pi_i:H^N\to H$ is the projection onto the $i^{th}$ component, for some $i\in \{1,\dots, N\}$. The fixed-point encoding $(T,S)$ is frugal and has state graph $G$ from Algorithm \ref{alg:splitting_1}, is unconditionally stable from Theorem \ref{thm:convergence_graph_drs} and has minimal variables as $\bD= H^{N-1}$.
	\end{proof}
	\begin{remark}\label{rem:independence_from_onto}
		Algorithm \ref{alg:splitting_1} does not depend on the onto decomposition of the Laplacian $L$ of the base graph. Given two different onto decompositions of $L$, say $L = ZZ^* = \widetilde{Z}\widetilde{Z}^*$ with $Z, \ \widetilde{Z}\in \mathbb{R}^{N\times (N-1)}$, from Proposition \ref{prop:uniqueness_onto}, there exists an orthogonal matrix $O\in \mathbb{R}^{(N-1)\times (N-1)}$ such that $\widetilde{Z} = ZO$. As before, let $\cZ, \ \widetilde{\cZ}, \ \mathcal{O}$ be the corresponding block operators on $H^{N}$ and $\{\widetilde{\bw}^k\}_k, \ \{\widetilde{\bx}^k\}_k$ and $\{\bw^k\}_k, \ \{\bx^k\}_k$ be the two sequences given by Algorithm \ref{alg:splitting_1} with respect to $\widetilde{Z}$ and $Z$, respectively. Note from \eqref{eq:proposed_reduced} that
		\begin{equation}\label{eq:independence_onto_1}
			\widetilde{\bx}^{k+1} = \left(\mathcal{L}+\mathcal{A}_L\right)^{-1}\widetilde{\mathcal{Z}}\widetilde{\bw}^k = \left(\mathcal{L}+\mathcal{A}_L\right)^{-1}\cZ \mathcal{O}\widetilde{\bw}^k,
		\end{equation}
		and that $\widetilde{\bw}^{k+1} =\widetilde{\bw}^k-\theta_k \mathcal{O}^*\cZ^*\widetilde{\bx}^{k+1}$, thus for every $k \in \mathbb{N}$,
		\begin{equation}\label{eq:independence_onto_2}
			\mathcal{O}\widetilde{\bw}^{k+1} =\mathcal{O}\widetilde{\bw}^k-\theta_k \cZ^*\widetilde{\bx}^{k+1}.
		\end{equation}
		From \eqref{eq:independence_onto_1} and \eqref{eq:independence_onto_2} we get that if $\bw^0 = \mathcal{O}\widetilde{\bw}^0$, then for every $k \in \mathbb{N}$, $\bw^k = \mathcal{O}\widetilde{\bw}^k$ and $\widetilde{\bx}^{k} = \bx^k$. Thus, the two choices generate the same sequences modulo an orthogonal transformation of the space.
	\end{remark}

	\subsection{Examples}\label{sec:examples}

	\paragraph{Tree base graphs.} Let $biG = (\cN, \cE, \cE')$ be a bilevel graph and assume that $G' = (\cN, \cE')$ defines a tree, i.e., $G'$ has no cycles. This case is particularly relevant for our work, because, for trees, an onto decomposition of the Laplacian of the base graph is simply given by the \emph{incidence matrix}, so that there is no need to factorize the Laplacian numerically.

	To define the incidence matrix of $G'$ we first need to order the edges from $1$ to $|\cE'|$. Then, we let $Z \in \mathbb{R}^{N\times |\cE'|}$ be the matrix such that $Z_{ij}=1$ if the $j^{th}$ edge leaves $i$,  $Z_{ij}=-1$ if the $j^{th}$ edge enters $i$, and $0$ otherwise. It is well known that if $G'$ is a connected tree, then $|\cE'|=N-1$ and thus the incidence matrix is full-rank. Furthermore, $Z$ is such that $ZZ^*$ is equal to the graph Laplacian of $G'$, thus, in this way, $L=ZZ^*$ is indeed an onto decomposition with $Z\in \mathbb{R}^{N\times(N-1)}$.

	Choosing a tree as a base graph, for the sake of notation, we can associate the variables with the edges in the base graph and rename $w_j$ as $w_{{(h,i)}}$, where $(h,i)$ is the $j^{th}$ edge in $\mathcal{E}'$.

	Given a bilevel graph $biG=(\mathcal{N}, \mathcal{E}, \cE')$, with a tree base graph $G'=(\mathcal{N}, \mathcal{E}')$, a step-size $\sigma>0$ and relaxation parameters $\theta_k\in(0,2]$ such that $\sum_k\theta_k(2-\theta_k)=+\infty$, denoting by $d_1,\dots, d_N$ the degrees of the nodes in the state graph, Algorithm \ref{alg:splitting_1} turns into the FRS method with minimal variables as presented in Algorithm \ref{alg:algorithm_trees}.

	\begin{algorithm}[H]\label{alg:algorithm_trees}

		\textbf{Initialize:} $w_{(h,i)}^0\in H$ for all $(h,i)\in \mathcal{E}'$

		\For{$k=0, 1, \dots$}{
			\For{$i=1,\dots, N$}{
				\begin{equation}\label{eq:tree_bg_update_x}
					x_i^{k+1} = J_{\frac{\sigma}{d_i}A_i}\bigg(\frac{2}{d_i}\sum_{(h,i)\in \cE}x_h^{k+1} +  \frac{1}{d_i}\bigg[\sum_{(i,j) \in \mathcal{E}'} w_{(i,j)}^{k}-\sum_{(h,i) \in \mathcal{E}'} w_{(h,i)}^k\bigg]\bigg)
				\end{equation}
			}
			\For{$(h,i)\in \cE'$}{
				\begin{equation}\label{eq:tree_bg_update_w}
					w_{(h,i)}^{k+1} = w_{(h,i)}^k + \theta_k(x_i^{k+1}-x_h^{k+1})
				\end{equation}
			}
		}
		\caption{Graph-based Douglas--Rachford with a tree base graph.}
	\end{algorithm}

	\paragraph{Ryu's splitting for $N$ operators.} To generalize the method introduced by Ryu in \cite{Ryu} for the $3$-operator problem to the $N$-operator problem, we consider a complete state graph with a star-shaped tree base graph having node $N$ as the root. In this way, the base graph edge set consists of $\cE' = \{(i,N) \mid i = 1,\dots, N-1\}$. With this choice and turning back again to the standard indexing of the  $w$ variables, we get
	\begin{equation*}
		\left\{
		\begin{aligned}
			x_i^{k+1}   & = J_{\frac{\sigma}{N-1}A_i}\bigg(\frac{2}{N-1}\sum_{h=1}^{i-1}x_h^{k+1} +  \frac{1}{N-1}w_i^k\bigg) \ \text{for} \ i\in \{1,\dots, N-1\}, \\
			x_N^{k+1}   & = J_{\frac{\sigma}{N-1}A_N}\bigg(\frac{2}{N-1}\sum_{h=1}^{N-1}x_h^{k+1} - \frac{1}{N-1}\sum_{j=1}^{N-1} w_j^k\bigg),                      \\
			w_{j}^{k+1} & = w_{j}^k + \theta_k(x_N^{k+1}-x_j^{k+1}) \ \text{for} \ j \in \{1,\dots, N-1\}.
		\end{aligned}
		\right.
	\end{equation*}
	\noindent Note, in fact, that if $N=3$, after a suitable rescaling of the $w$ variables, one gets the method in \cite[Section 4.1]{Ryu}. This is one possible way to correct the attempt in \cite[Remark 4.7]{malitsky2021resolvent}.

	\paragraph{Malitsky--Tam splitting as a proximal point method.} If we consider a sequential tree base graph, i.e., $\cE' :=\{(i,i+1) \mid i = 1, \dots, N-1\}$, and the state graph $\cE = \cE'\cup \{(1,N)\}$, Algorithm \ref{alg:algorithm_trees} turns into the Malitsky--Tam splitting introduced in \cite{malitsky2021resolvent}. We can therefore conclude that both the Ryu and the Malitsky--Tam splitting can be understood as proximal point methods, which answers an open question in the conclusions of \cite{malitsky2021resolvent}.

	\paragraph{Three operator splitting with complete base graph.} The proposed graph-based DRS method when applied to the three operator problem encompasses several already known extensions of the DRS method, namely: the sequential extension introduced in \cite{bredies2021degenerate} (choosing $\cE=\cE'=\{(1,2), (2,3)\}$), two well-known parallel extensions \cite[Section 9.1]{condat} (choosing $\cE=\cE' = \{(1,3), (2,3)\}$ or $\cE=\cE'=\{(1,2), (1,3)\}$), and the Ryu and the Malitsky--Tam methods (choosing, respectively, $\cE'=$ $\{(1,3)$ $,(2,3)\}$, $\cE=\{(1,3), (2,3), (1,2)\}$ and $\cE' = \{(1,2), (2,3)\}$, $\cE=\{(1,3), (2,3), (1,3)\}$). Note in particular that all the aforementioned choices feature tree base graphs. Yet, for the $3$-operator problem the proposed framework provides another interesting architecture: The case of a complete base graph. Here, as an onto decomposition of the Laplacian of the complete graph we can pick
	\begin{equation*}
		Z^* = \begin{bmatrix}
			\sqrt{2} & -\sqrt{1/2} & -\sqrt{1/2} \\
			0        & \sqrt{3/2}  & -\sqrt{3/2}
		\end{bmatrix}.
	\end{equation*}
	With this choice, considering $\widetilde{w}_1^k= \sqrt{2} w_1^k$ and $\widetilde{w}_2^k = \sqrt{2/3}w_2^k$ for every $k \in \mathbb{N}$, Algorithm \ref{alg:splitting_1} writes
	\begin{equation*}
		\left\{
		\begin{aligned}
			x_1^{k+1}             & = J_{\frac{\sigma}{2}A_1}\left(\tfrac{1}{2}\widetilde{w}_1^k\right),                                                                             \\
			x_2^{k+1}             & = J_{\frac{\sigma}{2}A_2} \left(x_1^{k+1}+\tfrac{3}{4}\widetilde{w}_2^k-\tfrac{1}{4}\widetilde{w}_1^k\right),                                    \\
			x_3^{k+1}             & =J_{\frac{\sigma}{2}A_3}\left(x_1^{k+1}+x_2^{k+1}-\tfrac{3}{4}\widetilde{w}_2^k-\tfrac{1}{4}\widetilde{w}_1^k\right),                            \\
			\widetilde{w}_1^{k+1} & = 	\widetilde{w}_1^k-\theta_k(2x_{1}^{k+1}-x_{2}^{k+1}-x_3^{k+1}), \quad \widetilde{w}_2^{k+1} = \widetilde{w}_2^k-\theta_k(x_2^{k+1}-x_3^{k+1}),
		\end{aligned}
		\right.\vspace{0.2cm}
	\end{equation*}
	where $\sigma>0$ is a positive step-size and $\{\theta_k\}_k$ in $(0,2]$ are positive relaxation parameters. Note that in Section \ref{sec:experiments}, we will see that considering fully connected base graphs can lead to faster methods.

	\subsection{Further properties}\label{sec:further_properties}

	The reduced PPP method, and in particular Algorithm \ref{alg:splitting_1}, inherits the well known general convergence guarantees of the proximal point algorithm. Indeed, if we assume that the relaxation parameters satisfy $0<\inf_k \theta_k \leq \sup_k\theta_k<2$, then a standard result is the asymptotic rate
	\begin{equation}\label{eq:rate_iterate}
		\|\widetilde{\mathcal{T}}\boldsymbol{w}^k-\boldsymbol{w}^k\|^2= o(k^{-1}) \quad \text{for $k \to +\infty$},
	\end{equation}
	see e.g., \cite[Theorem 1]{Davis2016}. Note that, as we set $\theta_k\in[\epsilon,2-\epsilon]$ for some $0<\epsilon <1$, then also $\|\boldsymbol{w}^{k+1}-\boldsymbol{w}^k\|^2=o(k^{-1})$. In our framework, the residual $\|\widetilde{\mathcal{T}}\bw^{k}-\bw^k\|^2$ has an elegant connection with an interesting quantity that measures how far the solution estimates $x_i^{k+1}$ are from consensus, namely, the \emph{state variance}, which we define for all $k \in \mathbb{N}$ as
	\begin{equation}\label{eq:state_variance}
		\Var(\bx^k):=\frac{1}{N}\sum_{i=1}^N \| x^{k}_i - \bar{x}^k \|^2, \quad \text{for} \ \bar{x}^k:= \frac{1}{N} \sum_{i=1}^N x_i^k.
	\end{equation}
	Indeed, we have the following result.
	\begin{prop}\label{prop:variance_residual}
		Let $N\geq 2$, let $biG = (\cN, \cE, \cE')$ be a bilevel graph for the $N$-operator problem \eqref{eq:Nop} for $(A_1,\dots, A_N)\in \MM_N$ with $\zer (A_1+\cdots+A_N)\neq \emptyset$. Let $\boldsymbol{x}^{k}=(x_1^k, \dots, x_N^k)$, $\boldsymbol{w}^k=(w_1^k,\dots, w_{N-1}^k)$ be the sequences generated by Algorithm \ref{alg:splitting_1} with step-size $\sigma>0$ and relaxation parameters $\{\theta_k\}_k$ in $(0,2]$ such that $\sum_{k=0}^{\infty}\theta_k(2-\theta_k)=+\infty$. Then
		\begin{equation}\label{eq:inequality_variance_residual}
			\Var(\bx^{k+1}) \leq  \frac{1}{\lambda_1N} \|\widetilde{\mathcal{T}}\bw^{k}-\bw^{k}\|^2 \quad \text{for all $k \in \mathbb{N}$},
		\end{equation}
		where $\lambda_1$ is the \emph{algebraic connectivity} of the base graph, i.e., the first nonzero eigenvalue of the graph Laplacian.
	\end{prop}
	\begin{proof}
		From $\widetilde{\mathcal{T}}\bw^k=\bw^k-\mathcal{Z}^*\bx^{k+1}$ for all $k \in \mathbb{N}$, we deduce
		\begin{equation}\label{eq:rate_Z*x_discussion}
			\| \mathcal{Z}^* \bx^{k+1} \|^2 = \|\widetilde{\mathcal{T}}\bw^k-\bw^k\|^2.
		\end{equation}
		Since $\ker\mathcal{Z}^*=\{(x, \dots, x) \ | \  x \in H\}$, the projection onto the orthogonal complement of $\ker\mathcal{Z}^*$ of $\bx^{k+1}$ is simply $\bx^{k+1}- \bar{\bx}^{k+1}$ where $\bar{\bx}^{k+1} = (\bar{x}^{k+1}, \dots, \bar{x}^{k+1})\in H^{N}$, which, since $\mathcal{L} = \cZ\cZ^*$, gives
		\begin{align}\label{eq:variancebound_discussion}
			\lambda_1\|\bx^{k+1}-\bar{\bx}^{k+1}\|^2 & \leq  \langle \mathcal{L}(\bx^{k+1}-\bar{\bx}^{k+1}), \bx^{k+1}-\bar{\bx}^{k+1}\rangle \nonumber \\
			                                         & =  \|\mathcal{Z}^*(\bx^{k+1}-\bar{\bx}^{k+1})\|^2 = \|\cZ^*\bx^{k+1}\|^2.
		\end{align}
		The identity \eqref{eq:rate_Z*x_discussion} together with \eqref{eq:variancebound_discussion} yields \eqref{eq:inequality_variance_residual}.
	\end{proof}

	Note that Proposition \ref{prop:variance_residual} shows an interesting dependence on the \emph{algebraic connectivity} of the base graph, which we shall investigate better in the experiments (cf., Section \ref{sec:congestedtransport}). Additionally, from \eqref{eq:variancebound_discussion} we can already conclude that the state variance converges to zero with a $o(k^{-1})$ worst-case rate.

	We now address the converse question, that is, whether the state variance can be understood as a measure of convergence for Algorithm \ref{alg:splitting_1}. Here, another feature of the underlying bilevel graph turns out to be particularly relevant, namely, the \emph{unbalance} of the state graph:
	\begin{equation}
		U_G:=\sqrt{\frac{1}{N}\sum_{i=1}^N(d_i^--d_i^+)^2}
	\end{equation}
	where, for all $i\in \{1,\dots, N\}$, $d_i^+$ and $d_i^-$ denote the out-degree and the in-degree of node $i$ in $G$ respectively. We observe, first, that from the maximality of $A_1, \dots, A_N$ in \eqref{eq:Nop}, Algorithm \ref{alg:splitting_1} uniquely defines $N$ sequences $\{a_i^{k+1}\}_k$, with $a_i^{k+1} \in A_ix_i^{k+1}$ for $i \in \{1,\dots, N\}$ and all $k \in \mathbb{N}$. By construction, $\boldsymbol{a}^{k+1} := (a_1^{k+1}, \dots, a_N^{k+1})\in H^{N}$ is characterized as the unique element in $\boldsymbol{A}\bx^{k+1}$ that solves
	\begin{equation}\label{eq:definition_of_ai}
		\mathcal{L}\bx^{k+1} + \left(\boldsymbol{\Sigma}+ \mathcal{P}\right)\bx^{k+1} + \sigma\boldsymbol{a}^{k+1} = \mathcal{Z}\bw^{k} \quad \text{for all $k \in \mathbb{N}$}.
	\end{equation}
	In the following result, we show that the state variance provides an upper bound for the norm of $\sum_{i=1}^Na_i^{k+1}$.

	\begin{prop}\label{prop:rate_ai}
		Let $N\geq 2$, let $biG = (\cN, \cE, \cE')$ be a bilevel graph for the $N$-operator problem \eqref{eq:Nop} for $\mathbf{A}=(A_1,\dots, A_N)\in \MM_N$ with $\zer (A_1+\cdots+A_N)\neq \emptyset$. Let $\boldsymbol{w}^k=(w_1^k,\dots, w_{N-1}^k)$, $\bx^{k+1}=(x_1^{k+1},\dots, x_N^{k+1})$ and $\boldsymbol{a}^{k} = (a_1^{k}, \dots, a_N^{k})\in \boldsymbol{A}\bx^{k}$ be the sequences generated by Algorithm \ref{alg:splitting_1} with step-size $\sigma>0$ and relaxation parameters $\{\theta_k\}_k$ in $(0,2]$ such that $\sum_{k=0}^{\infty}\theta_k(2-\theta_k)=+\infty$. Then,
		\begin{equation}\label{eq:gradient_variance}
			\bigg\|\sum_{i=1}^N a_i^k\bigg\|^2 \leq \  \frac{U_G^2 N^2}{\sigma^2} \Var(\bx^k) \quad \text{for all $k \in \mathbb{N}$}.
		\end{equation}
		In particular, if $x_1^k = \dots = x_N^k =: x^*$ for some $k\in \mathbb{N}$, then $x^*$ is a solution to \eqref{eq:Nop}.
	\end{prop}
	\begin{proof}
		First, notice that by definition of $\Sigma$ and $P$, for all $j\in \{1,\dots, N\}$, we have
		\begin{align*}
			\sum_{i = 1}^N(\Sigma +P)_{ij} & = \bigg[\sum_{i>j} L_{ij}\bigg]-\bigg[\sum_{i<j} L_{ij}\bigg] + \sum_{i\geq j}\sum_{(h,k)\in \cE\setminus \cE'} P^{hk}_{ij}                                                         \\
			                               & = \bigg[\sum_{(j,i)\in \cE'}(-1)\bigg]+\bigg[\sum_{(i,j)\in \cE'}1\bigg]+\bar{d_{j}}+\sum_{(j,h)\in \cE\setminus \cE'} (-2) = (d_j')^{+}-(d_j')^{-}+(\bar{d_j})^{+}-(\bar{d_j})^{-} \\
			                               & = d_j^+-d_j^-,
		\end{align*}
		where $(d_j')^+, \ (d_j')^{-}, \ (\bar{d_j})^+, \ (\bar{d_j})^-, \ \bar{d_j}, \ d_j^+, \ d_j^-$ are respectively: the in- and out-degrees of node $j$ in $G'$, the in- and out-degrees of node $j$ in $(\cN, \cE\setminus \cE')$, the degree of node $j$ in $(\cN, \cE\setminus \cE')$, and the in- and out-degrees of node $j$ in $G$. Therefore, we have
		\begin{equation}\label{eq:unbalance_formulations}
			\|(\boldsymbol{1}\otimes I)^*(\boldsymbol{\Sigma} + \mathcal{P})\|=\sqrt{N} U_G.
		\end{equation}
		Since $\{(x, \dots, x) \mid  x \in H\} \subset \ker (\boldsymbol{1}\otimes I)^*(\boldsymbol{\Sigma}+\mathcal{P})$, using \eqref{eq:unbalance_formulations}, we get
		\begin{align*}
			\|(\boldsymbol{1}\otimes I)^*(\boldsymbol{\Sigma}+\mathcal{P})\boldsymbol{x}^k\| & =\|(\boldsymbol{1}\otimes I)^*(\boldsymbol{\Sigma}+\mathcal{P})(\boldsymbol{x}^k-\bar{\boldsymbol{x}}^k)\| \\ &\leq \|(\boldsymbol{1}\otimes I)^*(\boldsymbol{\Sigma}+\mathcal{P})\|\|\boldsymbol{x}^k-\bar{\boldsymbol{x}}^k\| = U_G\sqrt{N} \|\boldsymbol{x}^k-\bar{\boldsymbol{x}}^k\|.
		\end{align*}
		From \eqref{eq:definition_of_ai}, since $\mathbf{1}^*Z=0$ by construction, $(\mathbf{1}\otimes I)^*\mathcal{Z}\bw = ((\mathbf{1}^* Z)\otimes I)\bw = 0$ for every $\bw\in \bD$, hence it easily follows that $\sigma\sum_{i=1}^{N}a_i^k=-(\boldsymbol{1}\otimes I)^*(\boldsymbol{\Sigma}+\mathcal{P}) \boldsymbol{x}^k$, and thus \eqref{eq:gradient_variance}. The rest of the proof is straightforward.
	\end{proof}

	Proposition \ref{prop:rate_ai} apart from answering the natural question on whether consensus is reached only at a solution point, which is trivial to prove but perhaps not clear from Algorithm \ref{alg:splitting_1}, also yields that if the state variance is small, i.e., all the solution estimates $x_1^{k}, \dots, x_N^{k}$ are close to consensus, then $\sum_{i=1}^N a_i^{k}$ is close to zero as well. Thus, without further structure, the state variance can actually be considered as a residual for Algorithm \ref{alg:splitting_1}.

	Gathering Propositions \ref{prop:variance_residual} and \ref{prop:rate_ai} with \eqref{eq:rate_iterate}, we can easily conclude that
	\begin{equation}\label{eq:summary_inequalities}
		\frac{\sigma^2}{N^2}\bigg\|\sum_{i=1}^N a_i^{k+1}\bigg\|^2 \leq  U_G^2  \Var(\bx^{k+1}) \leq \frac{U_G^2}{\lambda_1N} \| \widetilde{\mathcal{T}}\bw^{k}-\bw^{k}\|^2 = o(k^{-1}) \quad \text{for $k\to +\infty$}.
	\end{equation}
	Thus, also $\sum_{i=1}^Na_i^{k+1}$ converges to zero strongly as $k\to +\infty$ with a $o(k^{-1/2})$ rate.

	We conclude this section with a discussion on further results assuming additional hypotheses on the operators $A_1, \dots, A_N$. First, we assume that $A_i = \partial f_i$ for some convex, proper and lower semi-continuous functions $f_1, \dots, f_N:H\to \mathbb{R}\cup \{+\infty\}$, so that \eqref{eq:Nop} turns into
	\begin{equation}\label{eq:objective_problem}
		\text{find} \ x \in H \ \text{such that:} \ 0 \in (\partial f_1+\dots+\partial f_N)x
	\end{equation}
	Note that \eqref{eq:objective_problem} is equivalent to minimizing $f:=f_1+\dots+f_N$ under mild regularity conditions, see, e.g., \cite[Corollary 16.38]{BCombettes}, which we shall implicitly assume, and, in that case, the condition $\zer(A_1+\cdots+A_N)=\zer(\partial f_1+\cdots +\partial f_N)\neq \emptyset$, is equivalent to the existence of minimizers of $f$.
	\begin{prop}[Rate on the objective function]\label{prop:rate_objective}
		Let $N\geq 2$, let $biG = (\cN, \cE, \cE')$ be a bilevel graph for \eqref{eq:objective_problem} and assume that $\zer(\partial f_1+\cdots+\partial f_N)\neq \emptyset$. Let $\boldsymbol{x}^{k}=(x_0^k, \dots, x_N^k)$, $\boldsymbol{w}^k=(w_1^k,\dots, w_N^k)$ be the sequences generated by Algorithm \ref{alg:splitting_1} with step-size $\sigma>0$ and relaxation parameters $\{\theta_k\}_k$ in $[\epsilon, 2-\epsilon]$ for some $0<\epsilon <1$. Let $j\in \{1,\dots,N\}$ and suppose $f_i$ is locally Lipschitz continuous for all $i\neq j$, then we have
		\begin{equation}\label{eq:rate_objective_function}
			f(x_j^k)-\inf f = o(k^{-1/2}).
		\end{equation}
		Furthermore, if also $f_j$ is locally Lipschitz then $f(\bar{x}^k)-\inf f = o(k^{-1/2})$.
	\end{prop}
	\begin{proof}
		Let $x^*$ be a solution to \eqref{eq:Nop}. From convexity of $f_i$ we have for all $i\in \{1,\dots, N\}$ that
		\begin{equation}\label{eq:proof_objective}
			f_i(x_i^k)-f_i(x^*)\leq \langle a_i^k, x_i^k-x^*\rangle,
		\end{equation}
		where $a_i^k \in \partial f_i(x_i^{k})$ are defined by \eqref{eq:definition_of_ai}. Summing up for all $i\in \{1,\dots, N\}$ and using the local Lipschitz property of $f_i$ for all $i\neq j$, we get
		\begin{align}\label{eq:proof_objective2}
			f(x_j^k)-f(x^*) & =\sum_{i=1}^N\left( f_i(x_i^k)-f_i(x^*)+f_i(x_j^k)-f_i(x_i^k)\right) \nonumber \\ &\leq \sum_{i=1}^N \langle a_i^k, x_i^k-x^*\rangle + \sum_{i\neq j} L_i\|x_j^k-x_i^k\|,
		\end{align}
		where $L_i$ are the Lipschitz constants of $f_i$ for all $i\neq j$ on some ball containing $x_1^k,\dots,x_N^k$, for all $k\in \mathbb{N}$, which exists because $\{\boldsymbol{x}^k\}_k$ is a bounded sequence. Using \eqref{eq:summary_inequalities} three times together with the fact that $\{x_j^k-x^*\}_k$ and $\{a_i^k\}_k$ for all $i\neq j$ are bounded sequences (cf., \eqref{eq:definition_of_ai}), we have
		\begin{equation}\label{eq:obj_function_proof}
			\begin{aligned}
				\sum_{i=1}^N \langle a_i^k, x_i^k-x^*\rangle & = \bigg\langle \sum_{i=1}^N a_i^k, x_j^k-x^*\bigg\rangle+\sum_{i\neq j}\langle a_i^k,x_i^k-x_j^k\rangle  \\
				                                             & \leq \bigg\|\sum_{i=1}^N a_i^k \bigg\| \|x_j^k-x^*\|+\sum_{i\neq j}\|a_i^k\|\|x_i^k-x_j^k\|=o(k^{-1/2}).
			\end{aligned}
		\end{equation}
		Combining \eqref{eq:obj_function_proof} with \eqref{eq:proof_objective2} and again \eqref{eq:summary_inequalities} we get \eqref{eq:rate_objective_function}, since a solution of \eqref{eq:objective_problem} is also a zero for $\partial f$ and thus, a minimizer for $f$. The variance estimate in \eqref{eq:summary_inequalities} provides also the rate for $f(\bar{x}^k)-\inf f$ when $f_j$ is locally Lipschitz as well.
	\end{proof}

	Proposition \ref{prop:rate_objective} is in line with similar results on DRS and the more general forward--DRS scheme (see, e.g., \cite[Theorem 3.4, Corollary 3.5]{Davis_rates}). The next result shows that the convergence of the solution estimates is strong if at least one operator is \textit{uniformly monotone} (cf., \cite[Definition 22.1(iii)]{BCombettes}).
	\begin{prop}[Strong convergence]\label{prop:strongconvergence} Let $N\geq 2$, let $biG = (\cN, \cE, \cE')$ be a bilevel graph for the $N$-operator problem \eqref{eq:Nop} for $\mathbf{A}=(A_1,\dots, A_N)\in \MM_N$ with $\zer (A_1+\cdots+A_N)\neq \emptyset$. Assume that there exists $j \in \{1,\dots, N\}$ such that $A_j$ is uniformly monotone on every bounded set of $\dom \, A_j$. Then, for all $i \in \{1,\dots, N\}$ the sequences $\{x_i^k\}_{k}$ generated by Algorithm \ref{alg:splitting_1}, with step-size $\sigma>0$ and relaxation parameters $\{\theta_k\}_k$ in $[\epsilon, 2-\epsilon]$ for some $0<\epsilon <1$, converge strongly to a solution to \eqref{eq:Nop}.
	\end{prop}
	\begin{proof}
		Let $\{u^k\}_k$ the corresponding PPP sequence according to \eqref{eq:PPP} and $u^*=(\bx^*, \bv^*)\in\zer \, \mathcal{A}$ be the weak limit of $\{\mathcal{T}u^k\}_k$ (cf., proof of Theorem \ref{thm:convergence_graph_drs}), then, by Theorem \ref{thm:over_lifting_max_mon}, $\bx^*=(\mathbf{1}\otimes I)x^*$ for some $x^*$ that solves \eqref{eq:Nop}. Consider the bounded set $S = \{x^*\} \cup \{x_j^k\}_k \subset \dom \, A_j$. By definition of uniform monotonicity, there exists an increasing function $\phi~:~\mathbb{R}_+ \to [0,+\infty]$ that vanishes only at $0$ such that
		\[ \langle a_j-a_j',x-x' \rangle \geq \phi\left(\|x-x'\|\right) \quad \text{for all} \ x, x', a_j, a_j' \ \text{such that:} \ a_j\in A_jx,~ a_j'\in A_jx'.\]
		By definition of $\mathcal{T}$ we have $ \mathcal{M}(u^k-\mathcal{T} u^k) \in \mathcal{A}\mathcal{T} u^k $, and consequently, by construction,
		\begin{equation*}
			\mathcal{M}(u^k-\mathcal{T} u^k) =
			\begin{bmatrix}
				(\boldsymbol{\Sigma}+\mathcal{P})\bx^{k+1}+\sigma\boldsymbol{a}^{k+1}-\cZ \bv^{k+1} \\ \cZ^* \bx^{k+1}
			\end{bmatrix},
		\end{equation*}
		where $\boldsymbol{a}^{k+1} \in \boldsymbol{A}\bx^{k+1}$ and $(\bx^{k+1}, \bv^{k+1})=\mathcal{T}u^k$. Thus, since $u^* \in \zer \mathcal{A}$, there exists $\boldsymbol{a}^*\in \boldsymbol{A}\bx^*$ such that $(\boldsymbol{\Sigma}+\mathcal{P})\bx^*+\sigma\boldsymbol{a}^{*}-\cZ \bv^* = 0$ and $\cZ^* \bx^* = 0$. Therefore, we have for all $k \in \mathbb{N}$ that
		\begin{align*}
			\langle \mathcal{M}(u^k-\mathcal{T} u^k),\mathcal{T} u^k-u^*\rangle & =\langle (\boldsymbol{\Sigma}+\mathcal{P})\left(\bx^{k+1}-\bx^*\right)+\sigma(\boldsymbol{a}^{k+1}-\mathbf{a}^*)-\cZ (\bv^{k+1}-\bv^*) , \bx^{k+1}-\bx^* \rangle        \\
			                                                                    & \quad +\langle\cZ^* \bx^{k+1}, \bv^{k+1}-\bv^*\rangle                                                                                                                   \\
			                                                                    & =\langle (\boldsymbol{\Sigma}+\mathcal{P})\left(\bx^{k+1}-\bx^*\right), \bx^{k+1}-\bx^*\rangle+\langle\sigma(\boldsymbol{a}^{k+1}-\mathbf{a}^*), \bx^{k+1}-\bx^*\rangle \\
			                                                                    & \quad -\langle \bv^{k+1}-\bv^* , \cZ^*\bx^{k+1} \rangle +\langle\cZ^* \bx^{k+1}, \bv^{k+1}-\bv^*\rangle.
		\end{align*}
		Now, using that $\mathcal{P}$ is monotone by construction (cf., proof of Theorem \ref{thm:over_lifting_max_mon}) and $\boldsymbol{\Sigma}$ is skew-symmetric yields $\langle (\boldsymbol{\Sigma}+\mathcal{P})\left(\bx^{k+1}-\bx^*\right), \bx^{k+1}-\bx^*\rangle\geq 0$. Therefore, by the uniform monotonicity of $A_j$,
		\begin{align*}
			\langle \mathcal{M}(u^k-\mathcal{T} u^k),\mathcal{T} u^k-u^*\rangle & \geq \langle\sigma(\boldsymbol{a}^{k+1}-\mathbf{a}^*), \bx^{k+1}-\bx^*\rangle = \sum_{i=1}^N \sigma \langle a_i^{k+1}-a_i^*, x_i^{k+1}-x_i^*\rangle \\
			                                                                    & \geq \sigma \phi(\|x_j^{k+1}-x^* \|).
		\end{align*}
		Using \eqref{eq:rate_iterate}, recalling that $\mathcal{M}(u^k-\mathcal{T} u^k)=\mathcal{C}(\boldsymbol{w}^{k}-\widetilde{\mathcal{T}}\boldsymbol{w}^{k})$ and the fact that $\mathcal{T} u^k\rightharpoonup u^*$ weakly in $\bH$, the left hand-side vanishes as $k\to +\infty$, thus $x_j^k \to x$ strongly in $\bH$. For all other sequences $\{x_i^k\}_k$, \eqref{eq:summary_inequalities} yields in particular $x_i^k-x_j^k \to 0$ for all $i \in \{1,\dots, N\}$. The thesis follows.
	\end{proof}

	\begin{remark}
		Under the assumptions of Proposition \ref{prop:strongconvergence}, assuming further structure on the function $\phi$ given by the definition of uniform monotonicity, we shall also provide a rate for $\|x_j^k-x^*\|$. In particular, if there is some $\mu>0$ such that we can choose $\phi(h)=\mu h^2$ for all $h\in \mathbb{R}_+$, then we get $\|x_j^k-x^*\|^2=o(k^{-1/2})$ for $k \to +\infty$, which is in line with previous results on DRS and forward--DRS (see, \cite[Theorem 4.1.3]{Davis_rates}).
	\end{remark}

	\section{Application to distributed optimization}\label{sec:distributed}

	Let $biG=(\mathcal{N}, \mathcal{E}, \cE')$ be a bilevel graph for the $N$-operator problem \eqref{eq:Nop} and suppose that each node in $\cN$ represents an agent. Each agent is independent and can do parallel, asynchronous computations. Further, every agent can send and receive data to adjacent nodes, and has to wait for the transfer to finish in case of reception. The agents are supposed to collaborate to solve \eqref{eq:Nop}. For each $i \in \mathcal{N}$, we assume that only agent $i$ knows the operator $A_i$, and can access it only through evaluation of its resolvent. Additionally, we assume that for each variable $w_j$, there is only one agent in charge of storing and updating it.

	\subsection{Tree base graphs}\label{sec:tree_graphs}
	The inherent sparsity of the incidence matrix allows us to easily turn Algorithm \ref{alg:algorithm_trees} into a simple fully distributed scheme. Indeed, from Algorithm \ref{alg:algorithm_trees} it is easy to notice that agent $i$ at iteration $k \in \mathbb{N}$, in order to compute \eqref{eq:tree_bg_update_x}, would need to receive all the $x_h^{k+1}$ for every $(h,i)\in \cE$, which can be transmitted from $h$ to $i$, and all the variables $w_{(i,j)}^k$ and $w_{(h,i)}^k$ for all $(i,j)\in \cE'$ and $(h,i)\in \cE'$. Here, we can assume that each variable $w_{(h,i)}^k$ is stored and updated by agent $i$ for all $(h,i)\in \cE'$. In this way, agent $i$ would only need to receive, in addition to all $x_h^{k+1}$ from each $h$ with $(h,i)\in \cE$, all the $w_{(i,j)}^k$ from each $j$ with $(i,j)\in \cE'$. Eventually, once agent $i$ computed $x_i^{k+1}$, agent $i$ can also update all the variables $w_{(h,i)}^{k}$ for all $h$ such that $(h,i)\in \cE'$, as this computation only involves $w_{(h,i)}^k$ (already stored), $x_h^{k+1}$ (received from $h$) and $x_i^{k+1}$ (just computed). In summary, given a bilevel graph $biG = (\cN, \cE,\cE')$ with a tree base graph $G'=(\cN, \cE')$, in the notation of Algorithm \ref{alg:algorithm_trees}, we get the fully distributed protocol described in Algorithm \ref{alg:distributed_optimization}.

	\begin{algorithm}[H]\label{alg:distributed_optimization}

		\textbf{Initialize:} For each $(h,i)\in \mathcal{E}'$, agent $i$ chooses $w_{(h,i)}^0\in H$ and sends $w_{(h,i)}^0$ to each agent $h$ such that $(h,i)\in \cE'$

		\For{$k=0, 1, \dots$}{
		\For{$i \in \cN$}{
		Agent $i$ receives $w_{(i,j)}^k$ from each agent $j$ such that $(i,j)\in \mathcal{E}'$, and $x_h^{k+1}$ from each agent $h$ such that $(h,i)\in \mathcal{E}$\\
		Agent $i$ computes $x_i^{k+1}$ as in \eqref{eq:tree_bg_update_x}, and updates $w_{(h,i)}^{k+1}$ for all $(h,i)\in \mathcal{E}'$ as in \eqref{eq:tree_bg_update_w}\\
		Eventually, agent $i$ sends $x_i^{k+1}$ to each agent $j$ such that $(i,j) \in \mathcal{E}$, and $w^{k+1}_{(h,i)}$ to each agent $h$ such that $(h,i) \in \mathcal{E}'$
		}
		}
		\caption{Distributed protocol for the graph-based Douglas--Rachford method with a tree base graph.}
	\end{algorithm}

	\subsection{General base graphs}

	As we will show in the experiments, considering base graphs with high algebraic connectivity could be beneficial in terms of convergence speed. Here, though, the design of a distributed protocol with in some sense minimal information flow through the network can lead to severe graph-theoretical challenges, as one would need to find a sparse onto decomposition of the Laplacian of the base graph, which for large scale instances can be costly or even intractable \cite{Rose1972183,chordal_vandenberghe}.

	For general base graphs, we can avoid such sparse factorization issues at the cost of introducing one additional variable, ending up with $N$ variables instead of $N-1$. Specifically, considering the change of variables $\widetilde{\bw}^k=\cZ \bw^k$ for all $k \in \mathbb{N}$, the graph-based Douglas--Rachford iteration in \eqref{eq:proposed_reduced} reads for all $k\in \mathbb{N}$ as
	\begin{equation}\label{eq:proposed_reduced_Zw}
		\left\{
		\begin{aligned}
			\bx^{k+1}             & = \left(\mathcal{L}+\mathcal{A}_L\right)^{-1}\widetilde{\bw}^k, \\
			\widetilde{\bw}^{k+1} & = 	\widetilde{\bw}^k-\theta_k\mathcal{L}\bx^{k+1},               \\
		\end{aligned}
		\right.
		\quad \widetilde{\bw}^0=(\widetilde{w}^0_1,\dots, \widetilde{w}^0_N) \in \Img \cZ.
	\end{equation}

	\noindent In this setting, we assume that the variable $\widetilde{w}_i$ is stored and updated by the agent $i$, for every $i\in \{1,\dots, N\}$. The only additional restriction, here, is that the starting point $\widetilde{\bw}^0$ should be chosen in the range of $\cZ$, namely, such that $\widetilde{w}^0_1+\dots+\widetilde{w}^0_N = 0$. To avoid further communications between the agents, we could simply let the agents initialize $\widetilde{w}^0_i = 0$ for every $i\in \{1,\dots, N\}$. Setting, for all $i\in \cN$, $d_i'$ to be the degree of $i$ in $G'$, we get the distributed protocol described in Algorithm \ref{alg:distributed_general_graphs}.

	\begin{algorithm}[H]\label{alg:distributed_general_graphs}

		\textbf{Initialize:} For each $i\in \mathcal{N}$, agent $i$ chooses $\widetilde{w}_i^0 = 0 \in H$

		\For{$k=0, 1, \dots$}{
		\For{$i \in \cN$}{
		Agent $i$ receives $x_h^{k+1}$ from each agent $h$ such that $(h,i)\in \mathcal{E}$\\
		Agent $i$ computes $x_i^{k+1}$ as
		\begin{equation*}
			x_i^{k+1} = J_{\frac{\sigma}{d_i}A_i}\bigg(\frac{2}{d_i}\sum_{(h,i)\in \cE}x_h^{k+1} +  \frac{1}{d_i}\widetilde{w}_i\bigg)
		\end{equation*}\\
		Agent $i$ sends $x_i^{k+1}$ to each agent $j$ such that $(i,j)\in \cE$ and to each agent $h$ such that $(h,i)\in \cE'$\\
		Agent $i$ receives $x_j^{k+1}$ from all the agents $j$ such that $(i,j)\in \cE'$ and, eventually, updates $\widetilde{w}_i^{k+1}$ according to
		\begin{equation*}
			\widetilde{w}_i^{k+1} = \widetilde{w}_i^k -\theta_k\bigg(d_i' x_i^{k+1}-\sum_{j \in \text{adj}(i; G')} x_j^{k+1} \bigg)
		\end{equation*}
		}
		}
		\caption{Distributed protocol for the graph-based Douglas--Rachford method with a general base graph.}
	\end{algorithm}

	Note that Algorithm \ref{alg:distributed_general_graphs}, contrarily to Algorithm \ref{alg:algorithm_trees}, features two communication phases. The first communication phase is similar to the communication phase in Algorithm \ref{alg:algorithm_trees}, but only involves a transmission of the solution estimates. The second communication phase is fundamental, because to update $\widetilde{w}_i^{k+1}$, node $i$ needs: $\widetilde{w}_i^{k}$ (already stored), $x_i^{k+1}$ (computed before) and all the $x_j^{k+1}$ for all $j \in \text{adj}(i; G')$. The latter can be divided into two classes: $x_j^{k+1}$ with $j \leq i$, which node $i$ received in the first communication phase, and $x_j^{k+1}$ with $j\geq i$, received in the second communication phase. Note, in particular, that in Algorithm \ref{alg:distributed_general_graphs} only the solution estimates $x_1^{k}, \dots, x_N^{k}$ are shared.

	\section{Numerical experiments}\label{sec:experiments}

	In this section, we present our numerical implementation of the graph-based DRS and its distributed variant applied to a congested optimal transport problem and to a distributed Support Vector Machine problem. All the experiments are performed in Python on a Intel(R) Core(TM) i5-5200U CPU @ 2.20GHz and 8 Gb of RAM and are available for reproducibility at \url{https://github.com/TraDE-OPT/graph-DRS}.

	\subsection{Congested transport}\label{sec:congestedtransport}
	The congested transport problem has a rich history that dates back to the works of Beckmann in the 50's \cite{beckmann}. The problem gathered a renewed interest more recently thanks to its connections to the optimal transport theory, mainly due to Santambrogio, Carlier et al.~\cite{carlier_santambrogio}. The problem is of the form
	\begin{equation}\label{eq:congested_transport_problem}
		\min_{\sigma \in L^{3/2}(\Omega, \mathbb{R}^2)} \ \int_\Omega \|\sigma(x) \|^{3/2}dx +\int_\Omega \|\sigma(x) \|dx \quad \text{subject to:} \
		\bigg\{
		\begin{aligned}
			\Div \sigma = \nu-\mu & \ \text{in $\Omega$},          \\
			\sigma \cdot n = 0    & \ \text{on $\partial \Omega$},
		\end{aligned}
	\end{equation}
	where $\Omega$ is a smooth compact domain in $\mathbb{R}^2$, and $\mu, \nu$ are two (sufficiently regular) probability densities. The divergence constraint in \eqref{eq:congested_transport_problem} has to be understood weakly with Neumann boundary conditions ($n$ is the outward unit vector), see \cite[Section 4.4.1]{santambrogio2015optimal} for further details. A feasible $\sigma$ is referred to as \textit{transport flow} and its total variation $|\sigma|:\Omega \to [0, \infty)$, defined for all $x\in \Omega$ as $|\sigma |(x):=\|\sigma(x)\|$, is referred to as \textit{transport density}. To give a rough idea, the integral of the transport density on some region $A\subset \Omega$ can be understood as the total amount of mass that is moving from $\mu$ to $\nu$ passing through $A$, while the flow contains the information on the \textit{direction} that the mass is taking.

	Here, we suppose that the two probability densities $\mu$ and $\nu$ are separated by a region that does not allow any transportation, e.g., a lake, on top of which there is a bridge with limited capacity. Specifically, we introduce in \eqref{eq:congested_transport_problem} an additional constraint of the form $\| |\sigma|_{\texttt{Brg}}\|_\infty\leq C$ and $\sigma|_{\texttt{Wtr}}=0$ for some $C>0$, where $|\sigma|_{\texttt{Brg}}$ and $|\sigma|_{\texttt{Wtr}}$ are the restrictions of $|\sigma|$ on $\texttt{Brg}\subset \Omega$ (the bridge), and $\texttt{Wtr}\subset \Omega$ (the lake), respectively ($\texttt{Wtr}\cap \texttt{Brg}=\emptyset$). Confer to Figure \ref{fig:experiment_output} for an illustration.

	We discretize the problem on a square grid of size $n = p\times p$ (that we shall still denote by $\Omega$) using forward finite differences, getting to a problem of the form
	\begin{equation}\label{eq:cong_transp_disc}
		\min_{\sigma \in \mathbb{R}^{n\times 2}} \ \sum_{i=1}^n \|\sigma_i \|^{3/2} +\sum_{i=1}^n \|\sigma_i \| + \mathbb{I}\big\{\Lambda\sigma = \nu-\mu\big\} + \mathbb{I}\bigg\{ \
		\begin{aligned}
			\sigma_{|\texttt{Wtr}} & = 0,                                \\
			\|\sigma_{i}\|         & \leq C \ \forall i \in \texttt{Brg}
		\end{aligned} \  \bigg\}
	\end{equation}
	where \texttt{Brg} and \texttt{Wtr} denote the set of indices corresponding to the bridge and the lake, respectively, and $\Lambda$ is the discrete divergence operator with no-flux constraints on the boundary, and $\mathbb{I}\{x\in C\}$ denotes, with a slight abuse of notation, the indicator function of the convex set $C$, i.e., $\mathbb{I}\{x \in C\} = +\infty$ if $x\not\in C$ and $0$ else. The optimization problem \eqref{eq:cong_transp_disc} has four non-smooth \textit{simple} terms, i.e., for which one can easily compute the proximity operator and, thus, it is for us a good benchmark problem on which we can investigate all the different graph-based extensions of DRS method.

	\paragraph{The influence of the connectivity.} In this experiment, we investigate the influence of the algebraic connectivity of the base graph on the convergence performance of Algorithm \ref{alg:splitting_1}. As $N=4$, there are exactly $38$ possible state graphs. By enumeration, one can further see that the algebraic connectivity takes exactly four possible values, namely: $\lambda_1 = 2-\sqrt{2} \ (0.5857...), \ 1, \ 2$ and $4$.

	In our first numerical experiment, we set $p=35$, and consider bilevel graphs with equal state and base graphs. We pick a step-size $\tau=2$ and set $C = 5 \cdot 10^{-2}$. For each choice of the base graph $G'=(\cN, \cE')$, we run Algorithm \ref{alg:splitting_1} with step-size $\tau$, bilevel graph $biG = (\cN, \cE, \cE')$ with $\cE=\cE'$, and starting point $\bw^0=0$, and plot the state variance  \eqref{eq:state_variance} with respect to the iteration number with a specific color representing the algebraic connectivity of $G'$. The results are shown in Figure \ref{fig:connectivity}(a). Then, we repeat the procedure fixing a complete state graph and letting the base graph vary among all possible sub-graphs. The results are shown in Figure \ref{fig:connectivity}(b). Recall from Remark \ref{rem:independence_from_onto} that Algorithm \ref{alg:splitting_1} is in some sense independent on the onto decomposition of the Laplacian of the base graph, thus, we do not compare different factorization choices.

	\begin{figure}[ht]
		\centering
		\begin{subfigure}{0.45\textwidth}
			\centering
			\includegraphics[width=1\textwidth]{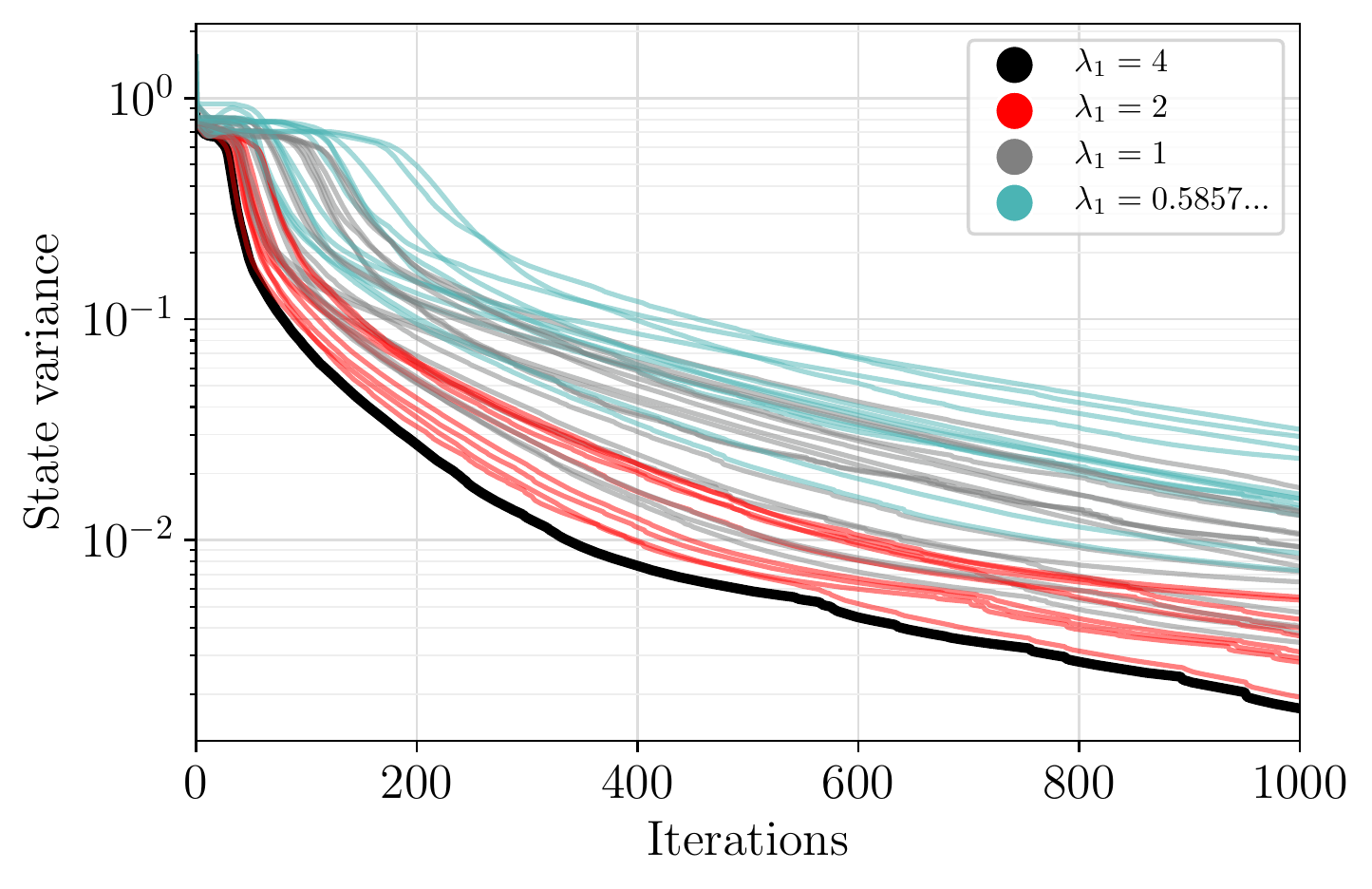}
			\caption{Setting the state graph $G$ equal to base graph $G'$ and letting $G$ vary.}
		\end{subfigure}
		\hfill
		\begin{subfigure}{0.45\textwidth}
			\centering
			\includegraphics[width=1\linewidth]{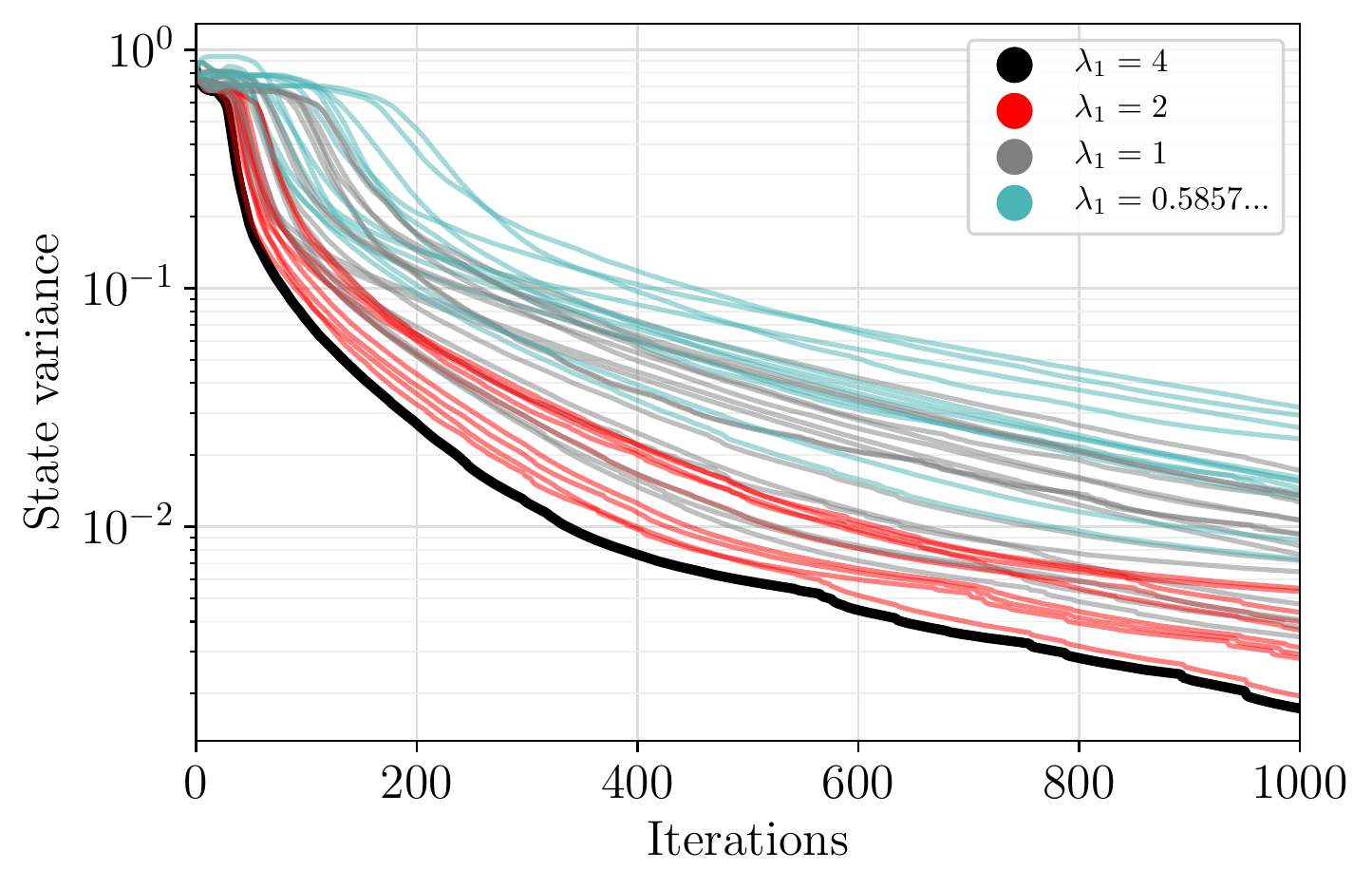}
			\caption{Setting the state graph $G$ to be the complete graph and letting $G'$ vary.}
		\end{subfigure}
		\caption{Influence of the algebraic connectivity $\lambda_1$ of the base graph on the decrease of the state variance  \eqref{eq:state_variance} as a function of the iteration number.}
		\label{fig:connectivity}
	\end{figure}

	In both cases, we can clearly see an effect of the algebraic connectivity of the base graph on the decrease of the state variance. Such a phenomenon is pretty common in the distributed optimization literature, see for instance \cite{Holly2021}. We plan to investigate it better in future works.

	\paragraph{The choice of the output.} In this experiment, we consider a much more refined grid ($p=720$), we set $G'$ (and thus $G$) to be the complete graph, pick the step-size $\tau = 10^{-1}$, and compare the four different estimates of the optimal solution to \eqref{eq:cong_transp_disc} yielded by Algorithm \ref{alg:splitting_1} before reaching consensus, which we denote by $\sigma_1^{k}, \dots, \sigma_4^{k}$ for every $k\in \mathbb{N}$. Specifically, $\sigma_1^k$ is associated to the divergence constraint, $\sigma_2^k$ to the superlinear term, $\sigma_3^k$ to the $\ell^1$ functional, and $\sigma_4^k$ to the bridge and water constraint (the last functional in \eqref{eq:cong_transp_disc}). As expected, in the red squares in Figure \ref{fig:experiment_output}(a) we can see that the solution estimate corresponding to the divergence constraint always satisfies the conservation law (i.e., the divergence constraint), while not respecting exactly the bridge constraint. The solution estimate corresponding to the bridge constraint presents the opposite behavior, and the one corresponding to the $\ell^1$ functional does provide a very \textit{sparse} estimate. These differences and an (approximately) optimal solution are shown in Figure \ref{fig:experiment_output}.

	\begin{figure}\label{fig:compare_outputs}

		\begin{subfigure}{0.49\textwidth}
			\centering
			\label{fig:compare_outputsa}
			\begin{tikzpicture}
				\node(a){\includegraphics[width=0.8\textwidth]{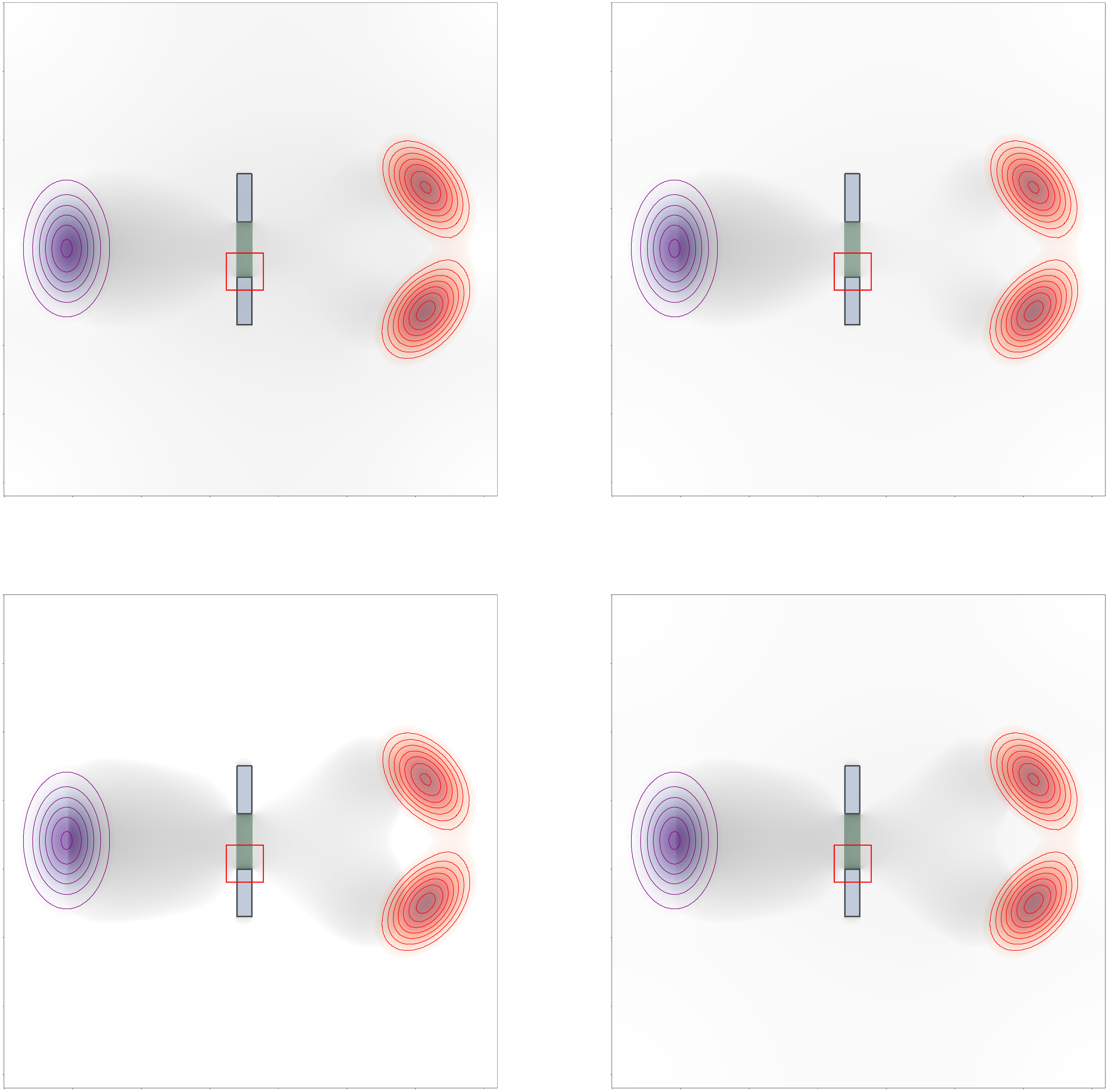}};
				\node at (a)
				[
					anchor=center,
					xshift=0mm,
					yshift=0mm
				]
				{
					\includegraphics[width=0.3\textwidth]{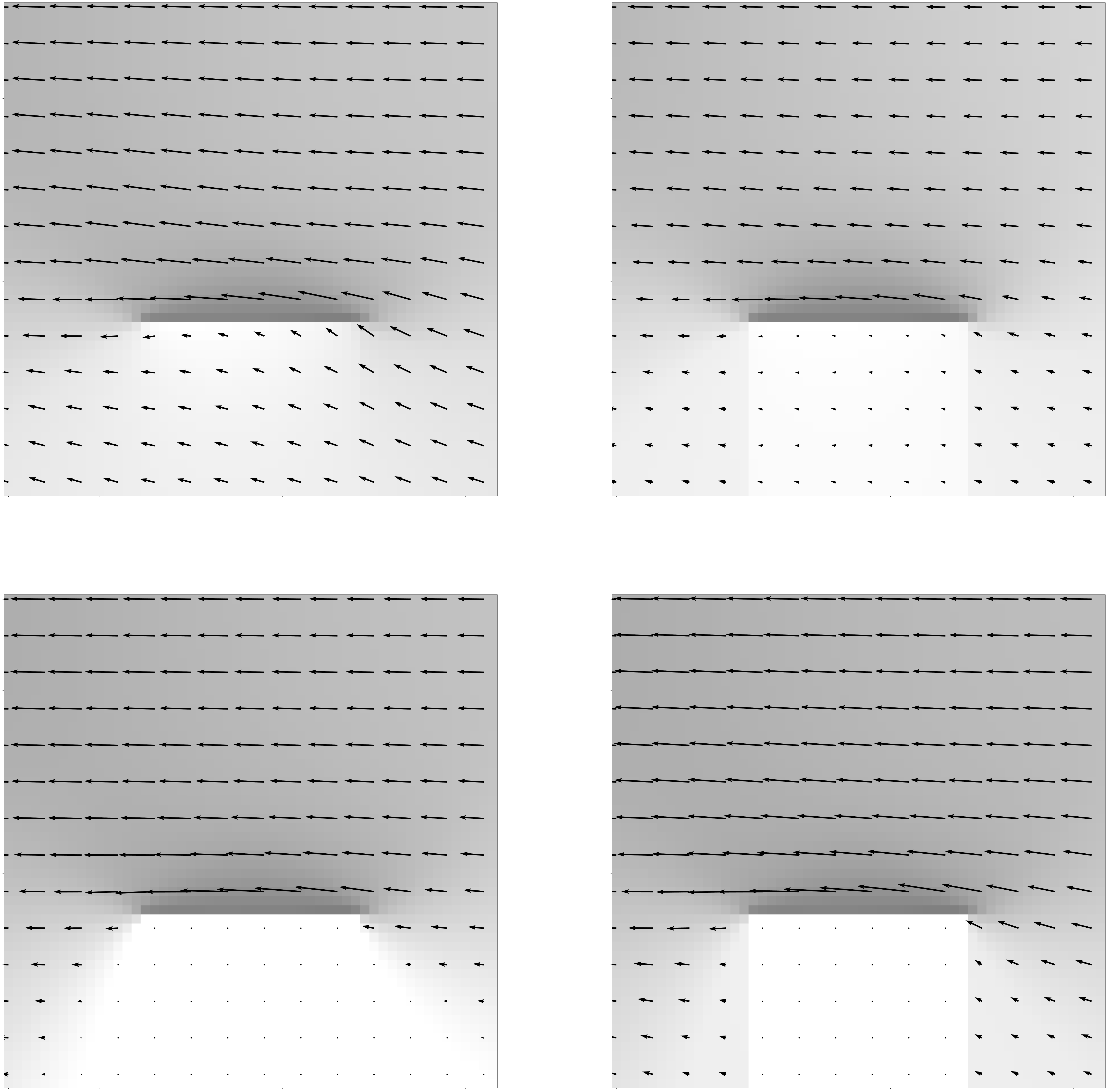}
				};
				\node at (-2.7,2.7) {$|\sigma_1^k|$};
				\node at (2.7,2.7) {$|\sigma_2^k|$};
				\node at (2.7,-2.7) {$|\sigma_4^k|$};
				\node at (-2.7,-2.7) {$|\sigma_3^k|$};
			\end{tikzpicture}
			\caption{Solution estimates at the early iteration $k=40$. The red squares correspond to the pictures displayed in the middle.}
		\end{subfigure}
		\hfill
		\begin{subfigure}{0.48\textwidth}
			\centering
			\includegraphics[width=0.81\linewidth]{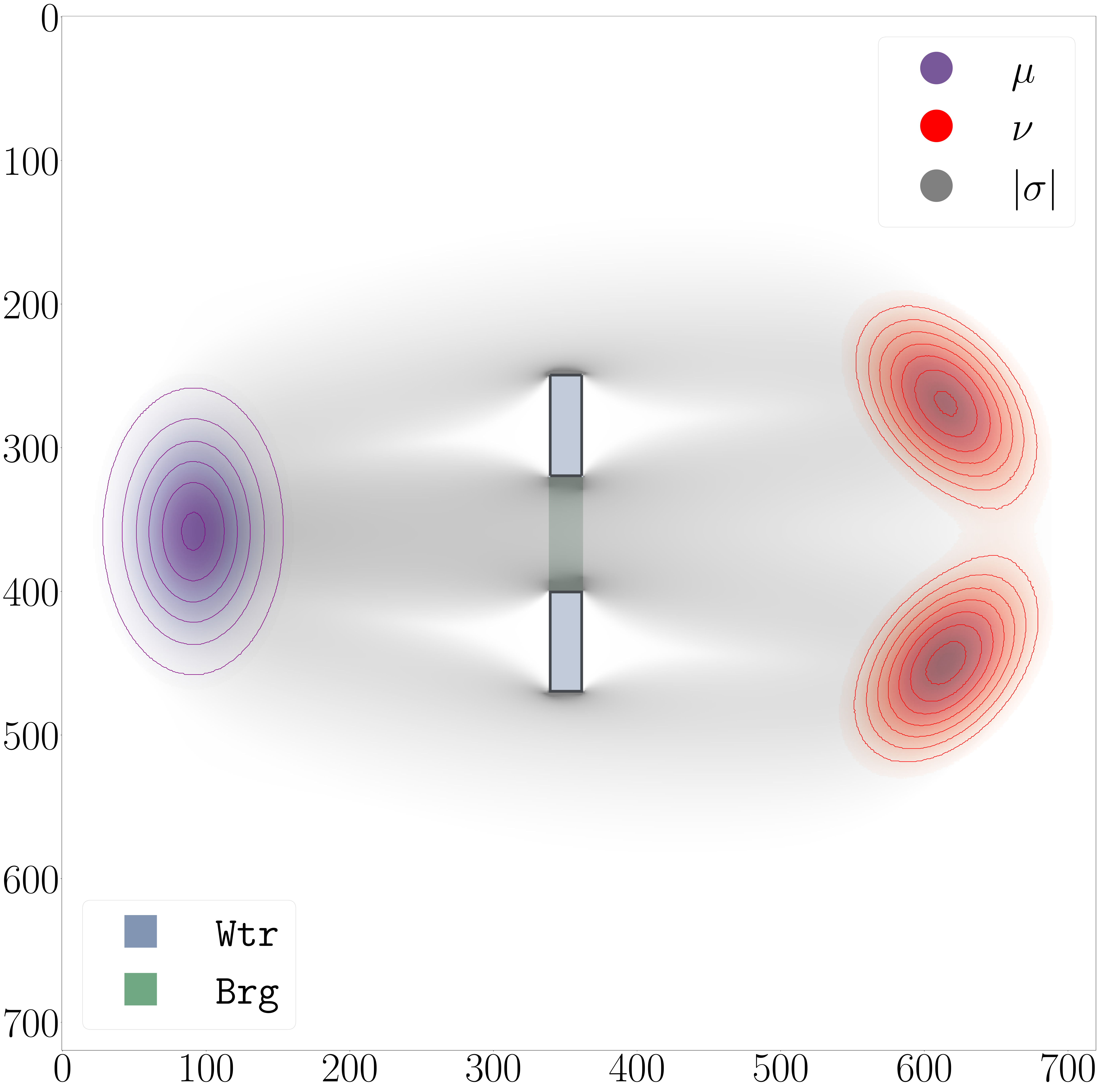}
			\caption{Solution estimate $\sigma_3^k$ after $300$ iterations and a state variance of $\Var(\boldsymbol{\sigma}^{k})=10^{-4}$, where $\boldsymbol{\sigma}^k=(\sigma_1^k, \dots, \sigma_4^k)$ is the vector of solution estimates.}
		\end{subfigure}
		\caption{Comparison of different output choices. To better visualize the vector fields $x \mapsto \sigma(x)$ we show in gray-scale the intensity of the corresponding transport densities, i.e., the scalar fields $|\sigma|:x \mapsto \| \sigma(x) \|$. The blue and the red regions correspond to the measures $\mu$ and $\nu$ respectively. Confer to Section \ref{sec:congestedtransport} for a detailed description.}
		\label{fig:experiment_output}
	\end{figure}
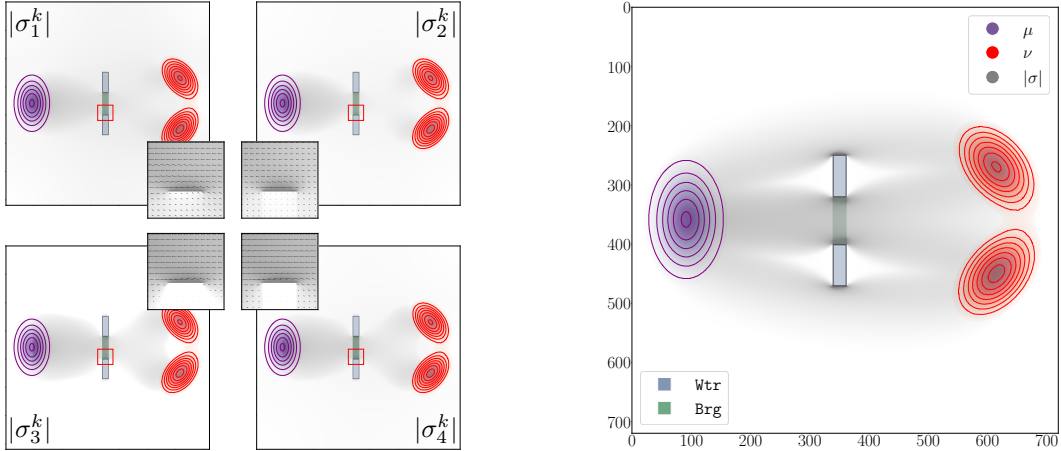

	\subsection{Distributed SVM}\label{sec:distributed_svm}

	In this experiment, we show an application of the proposed graph-based Douglas--Rachford method in a fully distributed optimization framework. We consider the classical \emph{Support Vector Machine} (SVM) problem formulated in primal form \cite{Chapelle2007}:
	\begin{equation}\label{eq:svm_infinite_dim}
		\min_{f \in \mathcal{H}_K} \ \sum_{i=1}^n \max\{1-y_if(x_i), 0\}+\gamma\|f\|_K^2,
	\end{equation}
	where $\{(x_1, y_1), \dots, (x_n, y_n)\}$ are labeled points in some domain $\Omega \subset \mathbb{R}^d$, $\gamma$ is a positive parameter, and $\mathcal{H}_K$ is a Reproducing Kernel Hilbert Space \cite{Aronszajn1950} endowed with the scalar product induced by the Gaussian kernel, namely $k(x,y) := \exp\{-\|x-y\|^2/(2\sigma^2)\}$ for some fixed $\sigma >0$ and all $x,y \in \Omega$.
	Despite having an infinite-dimensional formulation, by the Representer Theorem \cite{rep_thm}, an optimal solution of \eqref{eq:svm_infinite_dim} can be found as linear combination of the kernel function evaluated at the training points, namely $f^*(x) := \alpha_1^* k(x_1, x)+ \dots + \alpha_n^* k(x_n, x)$ for all $x \in \Omega$. Hence, problem \eqref{eq:svm_infinite_dim} admits a finite-dimensional reformulation, which reads as
	\begin{equation}\label{eq:distributed_SVM_objfun}
		\min_{\alpha \in \mathbb{R}^n} \ \sum_{i=1}^{n}\max\{1-y_i(k_i\cdot \alpha), 0\}+\gamma\alpha^*K\alpha,
	\end{equation}
	where $K\in \mathbb{R}^{n\times n}$ is the matrix defined as $K_{ij}:=k(x_i, x_j)$ for all $x_i, x_j$ in the training set, and $k_i\in \mathbb{R}^n$ is the $i^{th}$ row of $K$ for all $i\in \{1,\dots, n\}$.

	\paragraph{Privacy and communication constraints.} Suppose that each training point $x_i \in \Omega$ represents an agent equipped with a personal information, i.e., the label $y_i\in \{\pm 1\}$. Due to some security policy, each agent is able to communicate only with a \emph{local superior}, that we will refer to as \textit{official}. An official is an agent with a higher relevance in the network, who is in charge to treat the information of a subset of agents in its neighborhood. We assume that each official knows the full matrix $K$, whereas agent $i$ only knows its label $y_i$ and the $i^{th}$ row of $K$. We suppose there are $C$ officials spread around $\Omega$, which can communicate between each other. Of course, such a structure serves as an example and it has to be clear that, in general, any communication graph can be considered.

	Taking into account our communication constraints, we set: $g_c(\alpha):=d_c/(\sum_{c=1}^Cd_c) \alpha^* K \alpha$ for each $c \in \{1,\dots,C\}$, where $d_c$ are the degrees of the officials. Then, we partition the set of indices $i \in \{1, \dots, n\}$ into $C$ sets $\mathcal{I}_1, \dots, \mathcal{I}_C$, each containing exactly $p$ indices corresponding to the agents that communicate with the official $c$, and denote by $h_{c,i}(\alpha) :=\max\{1-y_\xi(k_\xi\cdot \alpha), 0\}$, where $\xi$ is the index of the $i^{th}$ agent under the official $c$. Therefore, the objective function \eqref{eq:distributed_SVM_objfun} can be split into
	\begin{equation}\label{eq:obj_fun_splitted}
		\min_{\alpha\in \mathbb{R}^n} \ \gamma \sum_{c=1}^C g_c(\alpha) +\sum_{c=1}^C\sum_{i=1}^p h_{c,i}(\alpha).
	\end{equation}
	In our model problem, we suppose that the officials are located in a circle in such a way that the first can communicate with the second and the last, while the second can only communicate with the third and the first, the third with the second and the fourth and so on. The resulting communication structure is depicted in Figure \ref{fig:communication_structure_SVM}. Specifically, we consider the ordered directed graph $G=(\cN, \cE)$ with nodes $\cN = \{1,\dots, C\}\cup (\{1,\dots, C\}\times \{1,\dots, p\})$, and edges $\cE := \{ (c,(c,i)) \mid c=1, \dots, C, \ i = 1, \dots, p \} \cup \{  (c,c+1) \mid c=1,\dots, C-1 \}\cup\{(1,C)\}$. Note that $\cN$ can be ordered enumerating the nodes as follows: $1, (1,1), (1,2), \dots, (1,p)$, $2, (2,1), (2,2),\dots, (2,p)$, $\dots$, $C, (C,1),\dots, (C,p)$. The terms in the objective function \eqref{eq:distributed_SVM_objfun} are associated with the nodes in $\cN$ in the obvious way (i.e., each function $g_c$ is associated to the node $c$, and each function $h_{c,i}$ is associated to the node $(c,i)$ for all $c=1,\dots, C$ and $i=1,\dots, p$), and, thus, can be ordered analogously.

	\begin{figure}
		\centering
		\begin{tikzpicture}
			\begin{scope}[every node/.style={circle,thick,draw, scale=0.025cm}]
				\node (C1) at (0,0) {$1$};
				\node (C2) at (3.5,0) {$2$};
				\node[draw=none] (CD) at (5.5,0) {\includegraphics[width=0.05\linewidth]{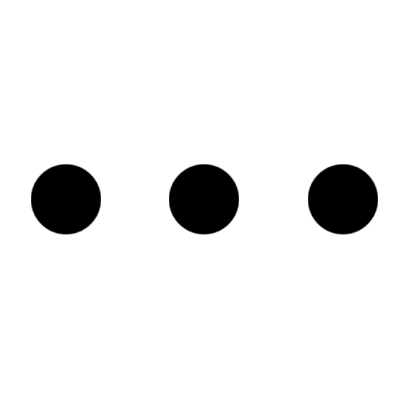}};
				\node (CC) at (7.5,0) {$C$};
				\node (A11) at (-1.2,-2) {$(1,1)$};
				\node (A12) at (-0.5,-2.6) {$(1,2)$};
				\node[draw=none] (A1D) at (0.5, -2.5) {\includegraphics[width=0.05\linewidth]{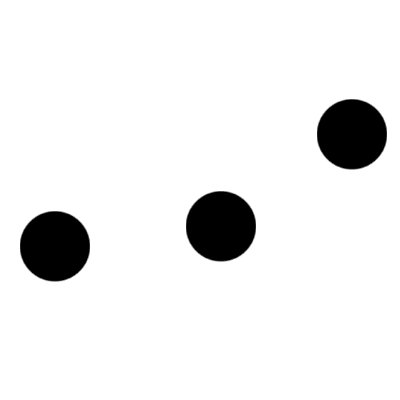}};
				\node (A1N) at (1.2,-2) {$(1,p)$};
				\node (A21) at (2.3,-2) {$(2,1)$};
				\node (A22) at (3,-2.6) {$(2,2)$};
				\node[draw=none] (A2D) at (4, -2.5) {\includegraphics[width=0.05\linewidth]{dots1.png}};
				\node (A2N) at (4.7,-2) {$(2,p)$};
				\node (AC1) at (6.3,-2) {$(C,1)$};
				\node (AC2) at (7.1,-2.6) {$(C,2)$};
				\node[draw=none] (ACD) at (8, -2.5) {\includegraphics[width=0.05\linewidth]{dots1.png}};
				\node (ACN) at (8.7,-2) {$(C,p)$};
			\end{scope}
			\begin{scope}[>={Stealth[black]},
				every node/.style={fill=white,circle},
				every edge/.style={draw=black,very thick}]
				\path  [->] (C1) edge[bend left=10] (C2);
				\path  [->] (C2) edge[bend left=20] (CD);
				\path  [->] (CD) edge[bend left=10] (CC);
				\path  [->] (C1) edge[bend right=10] (A11);
				\path [->] (C1) edge[bend right=10] (A12);
				\path [->] (C1) edge[bend left=10]  (A1D);
				\path [->] (C1) edge[bend left=10] (A1N);
				\path [->] (C2) edge[bend right=10] (A21);
				\path [->] (C2) edge[bend right=10] (A22);
				\path [->] (C2) edge[bend left=10]  (A2D);
				\path [->] (C2) edge[bend left=10] (A2N);
				\path [->] (CC) edge[bend right=10] (AC1);
				\path [->] (CC) edge[bend right=10] (AC2);
				\path [->] (CC) edge[bend left=10]  (ACD);
				\path [->] (CC) edge[bend left=10] (ACN);
			\end{scope}
			\begin{scope}[>={Stealth[black]},
				every node/.style={fill=white,circle},
				every edge/.style={draw=black}]
				\path [->] (C1) edge[bend left=30]  (CC);
			\end{scope}
		\end{tikzpicture}
		\caption{Bilevel graph considered in the distributed SVM experiment, cf., Section \ref{sec:distributed_svm}. Thick edges are associated to the base graph.}
		\label{fig:communication_structure_SVM}
	\end{figure}

	\paragraph{Methods and comparisons.} We compare the performance of the following methods.
	\begin{enumerate}
		\item A distributed Douglas--Rachford method according to Algorithm \ref{alg:distributed_optimization}, with step-size $\sigma>0$ and relaxation parameters $\theta_k=1$ for all $k\in \mathbb{N}$, associated to the bilevel graph $biG=(\cN, \cE, \cE\setminus \{(1,C)\})$, where $G=(\cN, \cE)$ is the ordered directed graph depicted in Figure \ref{fig:communication_structure_SVM}. Note in particular, that we are considering a tree base graph.

		\item P-EXTRA \cite[Algorithm 2]{PG-EXTRA} with step-size $\sigma>0$ and \emph{mixing matrices} $W=I-\tfrac{1}{n+C}L$ and $\widetilde{W}=\tfrac{1}{2}(W+I)$, where $L\in \mathbb{R}^{(n+C)^2}$ is the graph Laplacian of $G$ and $I\in \mathbb{R}^{(n+C)^2}$ is the identity.
		\item A distributed PDHG method \cite{Chambolle2011} obtained with the following procedure. Let $L$ be the Laplacian of the state graph in Figure \ref{fig:communication_structure_SVM}, and consider $\boldsymbol{L}:=L\otimes I_n$ where $I_n$ is the identity matrix in $\mathbb{R}^n$. The PDHG method can be used to solve in a distributed fashion the following product-space reformulation of \eqref{eq:distributed_SVM_objfun}:
		      \begin{equation}
			      \min_{\boldsymbol{\alpha}\in \mathbb{R}^{n\times(n+C)}} \ \gamma\sum_{c=1}^C g_c(\alpha_c) + \sum_{c=1}^C\sum_{i=1}^p h_{c,i}(\alpha_{c,i})+\mathbb{I}\{ \boldsymbol{L}\boldsymbol{\alpha} = \mathbf{0}\},
		      \end{equation}
		      where $\alpha_c$ and $\alpha_{c,i}$ for all $c\in \{1,\dots, C\}$ and $i\in \{1,\dots, p\}$ are all $n$-dimensional vectors and together form the $n\times (n+C)$ vector $\boldsymbol{\alpha}$. Note that $\boldsymbol{L}\boldsymbol{\alpha} = \mathbf{0}$ if and only if all these $\alpha$ coincide.
	\end{enumerate}

	To test and compare the three methods we use the following procedure. We create an artificial dataset in $\mathbb{R}^2$ with $n = 50$, $C=5$, and $p = 10$. We pick $10$ values for the step-size $\sigma > 0 $ in a logarithmic scale between $10^{-2}$ and $10^1$. For each step-size $\sigma$ we run the distributed DRS and P-EXTRA with step-size $\sigma$. For the PDHG method we only vary the primal step-size choosing the dual step-size $\gamma =(\sigma \|L\|^2)^{-1}$ to guarantee convergence. We tested PDHG also for other choices of dual step-size $\gamma \leq (\sigma \|L\|^2)^{-1}$, but the performance was always worse, and hence, these choices were discarded. At every iteration, every method provides an estimate of the optimal solution to \eqref{eq:distributed_SVM_objfun}, e.g., the mean of all estimated solutions of every single node. Such points are used to evaluate the objective function \eqref{eq:distributed_SVM_objfun} in Figure \ref{fig:comparisons}(a). Further, at every iteration we compute the three state variances, and, eventually, compare them in Figure \ref{fig:comparisons}(b).

	\begin{figure}[h!]
		\centering
		\begin{subfigure}{0.45\textwidth}
			\centering
			\includegraphics[width=1\textwidth]{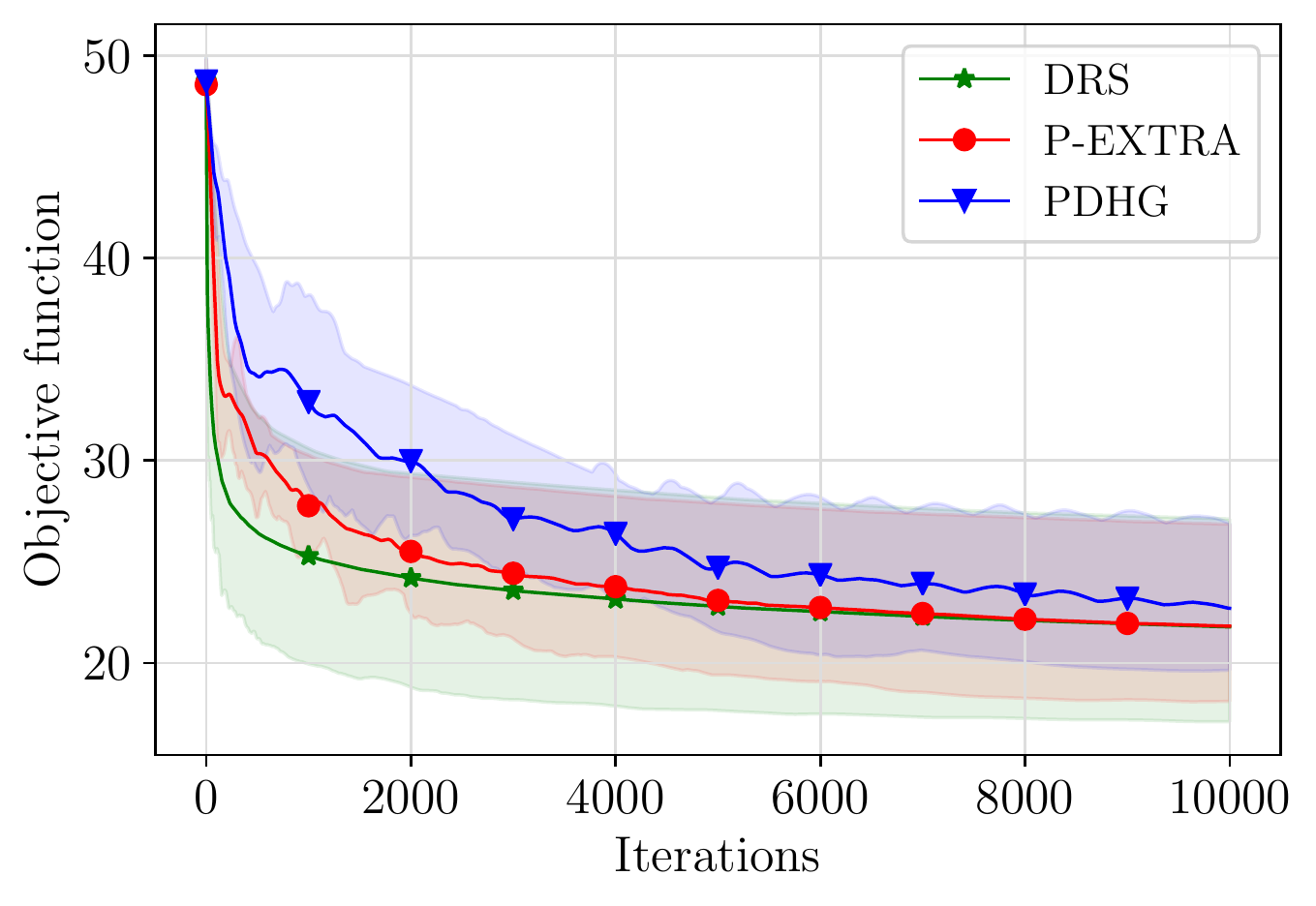}
			\caption{Mean objective function as a function of the iteration number for DRS, P-EXTRA and PDHG.}
		\end{subfigure}
		\hfill
		\begin{subfigure}{0.45\textwidth}
			\centering
			\includegraphics[width=1\linewidth]{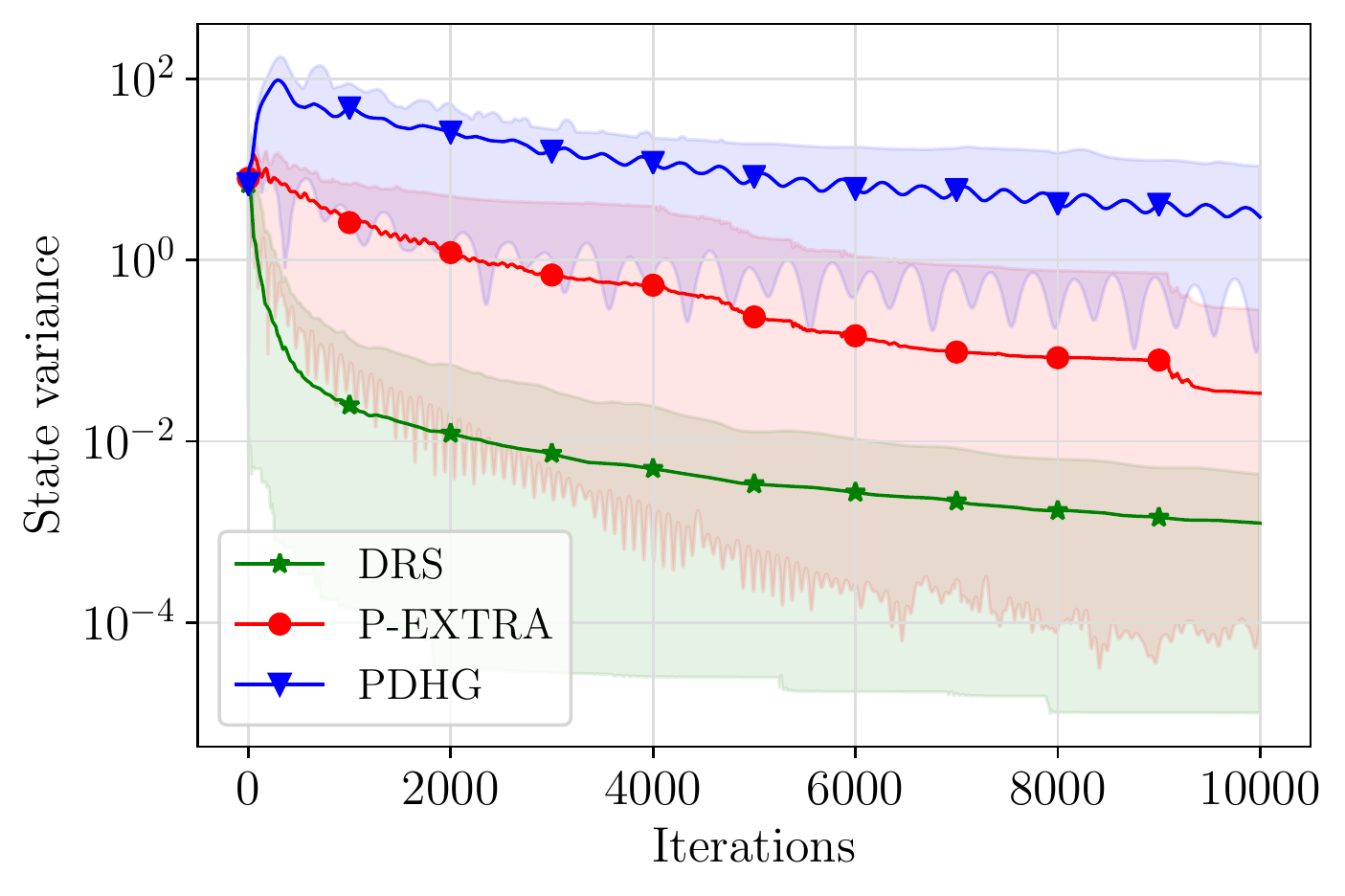}
			\caption{Mean state variance as a function of the iteration number for DRS, P-EXTRA and PDHG.}
		\end{subfigure}
		\caption{Comparison between the proposed method DRS, P-EXTRA and PDHG. For every method, each line is the mean of $10$ independent runs with different step-sizes. The corresponding shaded regions denote the corresponding best and the worst cases.}
		\label{fig:comparisons}
	\end{figure}

	\paragraph{Comments.} Figure \ref{fig:comparisons} shows the mean and the region between the best and the worst performance from $10$ independent runs with different step-size choices, for the three compared methods. For the sake of fairness, we compared the objective function and the state variance for the three methods as a function of the iteration number. Here, we can clearly see that the proposed distributed DRS outperforms P-EXTRA and the distributed PDHG, reaching a state variance of the order $10^{-2}$ within about a thousand of iterations, see Figure \ref{fig:comparisons}(b).

	\section{Conclusions}
	In this work, we proposed graph-based extensions of the DRS method based on the notion of bilevel graph, which encompasses several known and new generalizations of the DRS method to \eqref{eq:Nop}. This work shows that for the $N$-operator problem there are at least as many unconditionally stable FRS methods with a minimal lifting as the number of possible bilevel graphs for \eqref{eq:Nop}. In fact, we believe that these are infinitely many (modulo equivalence). A deeper question is whether the graph-based DRS encompasses \textit{all} possible methods for \eqref{eq:Nop}. This will be the topic of future work. In the future, we also plan to embed $N-1$ forward terms into Algorithm \ref{alg:splitting_1} leading to a graph-based extension of the Davis--Yin method.

\paragraph{Acknowledgments.} This work has received funding from the European Union’s Framework Programme for Research and Innovation Horizon 2020 (2014--2020) under the Marie Skłodowska-Curie Grant Agreement No. 861137. The Institute of Mathematics and Scientific Computing at the University of Graz, with which K.B.~and E.C.~are affiliated, is a member of NAWI Graz  (\url{https://nawigraz.at/en}).

\bibliography{biblio}
\bibliographystyle{ieeetr}

\end{document}